\documentclass[11pt]{amsart}
\usepackage[utf8]{inputenc}

\usepackage{amssymb,amsthm,amsmath,amstext}
\usepackage{eucal}

\usepackage{mathrsfs}
\usepackage{bm}
\usepackage[utf8]{inputenc}
\usepackage{comment}
\usepackage[all]{xy}

\usepackage[margin = 3cm]{geometry}

\theoremstyle{plain}
\newtheorem{theorem}{Theorem}[section]
\newtheorem*{conjecture*}{Conjecture}
\newtheorem{prop}[theorem]{Proposition}
\newtheorem{lemma}[theorem]{Lemma}
\newtheorem{coro}[theorem]{Corollary}
\theoremstyle{definition}
\newtheorem{remark}[theorem]{Remark}
\newtheorem*{ack}{Acknowledgements}

\def\CC{{\mathbb{C}}}

\def\PP{{\mathbb{P}}}
\def\QQ{{\mathbb{Q}}}\def\ZZ{{\mathbb{Z}}}

\def\cJ{{\mathcal{J}}}
\def\cO{{\mathcal{O}}}
\def\cE{{\mathcal{E}}}
\def\cA{{\mathrm{A}_C}}
\def\cF{{\mathcal{F}}}
\def\cG{{\mathcal{G}}}
\def\cI{{\mathcal{I}}}
\def\cK{{\mathrm{Kum}_C}}
\def\cL{{\mathcal{L}}}
\def\cP{{\mathcal{P}}}
\def\cM{{\mathcal{M}}}
\def\cN{{\mathcal{N}}}
\def\cR{{\mathcal{R}}}
\def\cU{{\mathcal{U}}}
\def\cC{{\mathcal{C}}}\def\cQ{{\mathcal{Q}}}
\def\cV{{\mathcal{V}}}\def\cF{{\mathcal{F}}}
\def\cW{{\mathcal{W}}}\def\cX{{\mathcal{X}}}
\def\ra{{\rightarrow}}
\def\lra{{\longrightarrow}}

\def\fe{{\mathfrak e}}
\def\fh{{\mathfrak h}}\def\fp{\mathfrak{p}}

\def\fsl{\mathfrak{sl}}
\def\fgl{\mathfrak{gl}}
\def\bS{S}

\DeclareMathOperator{\Tot}{Tot}

\DeclareMathOperator{\Hom}{Hom}

\DeclareMathOperator{\Pic}{Pic}

\DeclareMathOperator{\GL}{GL}
\DeclareMathOperator{\SL}{SL}
\DeclareMathOperator{\SU}{SU}
\DeclareMathOperator{\Ext}{Ext}

\usepackage[dvipsnames]{xcolor}



\title{The Coble quadric}

\author{V. Benedetti, M. Bolognesi, D. Faenzi, L. Manivel}

\usepackage{hyperref}
\begin{document}

\begin{abstract}
 Given a smooth genus three curve $C$, the moduli space of rank two stable vector bundles 
on $C$ with trivial determinant embeds in $\PP^8$ as a hypersurface whose singular locus is 
the Kummer threefold of $C$; this hypersurface is the Coble quartic. Gruson, Sam and Weyman 
realized that this quartic could be constructed from a general skew-symmetric four-form 
in eight variables. Using the lines contained in the quartic, we prove that a similar construction allows to recover $\SU_C(2,L)$, the moduli space 
of rank two stable vector bundles on $C$ with fixed determinant of odd degree $L$, as a subvariety of $G(2,8)$. In fact, each point $p\in C$ defines a natural embedding of $\SU_C(2,\cO(p))$ in $G(2,8)$. We show that, for the generic such embedding, 
there exists a unique quadratic section of the Grassmannian which is 
singular exactly along the image of $\SU_C(2,\cO(p))$, and thus deserves to be coined the Coble quadric of 
the pointed curve $(C,p)$.  
\end{abstract}

\maketitle

\section{Introduction}

A century ago, Arthur Coble proved 
that there exists a unique quartic hypersurface $\mathcal{C}$ in $\PP^7$ that is singular exactly along the 3 dimensional Kummer variety, image of the Jacobian of a genus $3$ curve $C$ via the $|2\Theta|$-linear system  (\cite{Coble}, see also \cite{beauville-coble, kollar}). This remarkable hypersurface is now named after him, and its many very special features have been studied by several algebraic geometers. For example $\mathcal{C}$ is projectively self-dual \cite{pauly}, it has close relationships with the $\Theta$-geometry of the curve (e.g. a Schottky-Jung configuration of Kummer surfaces of Prym varieties \cite{vGP}, etc.) and with moduli of configurations of points in the projective space \cite{AB15}. 

Probably, the most striking property is,  however, that $\mathcal{C}$ is the  image, via the theta map, of the moduli space of semi-stable rank two vector bundles on $C$ with trivial determinant. This was first remarked by Narasimhan and Ramanan in the seminal paper \cite{nr}. In particular, since the theta map is an embedding for rank two bundles with trivial determinant \cite{beauville1}, we can identify $\mathcal{C}$ with the moduli space $\SU_C(2)$ itself. 

In rank two there is, up to isomorphism, only one other moduli space $\SU_C(2,L)$ of rank two vector bundles on $C$, obtained by fixing the determinant to be any given line bundle $L$ of odd degree (up to non-canonical isomorphisms, $L$ is irrelevant). Contrary to $\cC$, this moduli space is smooth and we can wonder what could be an analogue of the 
Coble quartic. The main results of this paper answer this natural question. 

In order to achieve this, we will use the theory of theta representations \cite{vinberg},
in the way 
this was initiated in \cite{GSW} as a complex addition to arithmetic invariant theory. 
In our setting, the main point is that starting from the $\GL_8$-module $\wedge^4\CC^8$
one can easily construct the Coble quartics in terms of Pfaffian loci. From this point of 
view, the curve $C$ defined by a general element of $\wedge^4\CC^8$ is not immediately visible, but certain deep properties of the quartic $\cC$ become easy to establish. For example,
we give in Theorem \ref{selfdual} a short, self-contained proof of the self-duality of $\cC$.
Then we switch from $\PP^7$ to the Grassmannian $G(2,8)$ and observe that also in this 
Grassmannian, there exist natural Pfaffian loci corresponding to skew forms of rank at 
most $4$ and $6$, respectively of codimension $6$ and $1$:
$$D=D_{Z_6}(v)\subset Q=D_{Z_1}(v)\subset G(2,8).$$
Here $v$ is a general element in $\wedge^4 \CC^8$ and $Q$ is a quadric section of the Grassmannian that is singular exactly
along the six-dimensional smooth locus $D$ (the notation $D_{Z_i}(v)$ will be explained in Section \ref{sec_hecke}). The connection with the Coble quartic comes from the fact that $D$ parametrizes a family of lines on it, some of the so-called Hecke lines. We deduce (Theorem \ref{thm_main} later on):

\begin{theorem} $D\simeq \SU_C(2,L)$ for $L$ of odd degree. \end{theorem}

Consequently, the moduli space, which is smooth,  comes up with a natural hypersurface of which it is the singular locus, contrary to the even case for which the moduli space is singular and uniquely determined by its singular locus, which is the Kummer.
We extend the unicity statement by proving (Theorem \ref{Coble_G(2,8)} later on): 

\begin{theorem} $Q$ is the only quadratic section of the Grassmannian 
that is singular along $D$. \end{theorem}

Because of this property, $Q$ really deserves to be called a {\it Coble quadric}. Moreover, exactly as the Coble quartic, we show this hypersurface is self-dual in a suitable sense (Theorem \ref{selfdual2}). As a matter of fact, for each point $p\in C$, there is an embedding $$ \varphi_p : \SU_C(2,\cO_C(p)) \hookrightarrow G(2,8),$$ 
(see \cite{beauville}), and we show that at least for the generic $p$, there exists a unique quadric section of the Grassmannian that is singular along the moduli space (Theorem \ref{quademb}).

Remarkably, we found other instances of this phenomenon: for example, an eightfold inside the flag variety $Fl(1,7,8)$ whose singular locus is an abelian threefold, essentially the Jacobian of the curve (see Remark \ref{8fold}). 

\smallskip
The paper is organized as follows. In section 2, we recall a few classical 
results about lines on moduli spaces of vector bundles on curves, and more specifically
about lines in the Coble quartic. In section 3 we explain how the Coble quartic, 
the Kummer threefold and the associated Jacobian can be constructed from a skew-symmetric four-form in eight variables, and we give a short proof of the self-duality of the quartic. In section 4 we explain how this point of view allows to understand the lines in the Coble quartic in terms of orbital degeneracy loci \cite{bfmt, bfmt2}, and we deduce Theorem 1  (see Theorem 21). The 
resulting description as a relative Pfaffian locus makes it clear 
that the odd moduli space
is the singular locus of a special quadratic section of the Grassmannian $G(2,8)$. 
In order to prove that this special quadric is unique, we need to study the square of the ideal of the Grassmannian $G(2,6)$ in its Pl\"ucker embedding. Going back to the relative
setting we deduce Theorem 2 (see Theorem 27). We finally complete the picture by explaining why and how the special quadric is also self-dual.  

\begin{ack}
All authors partially supported by FanoHK ANR-20-CE40-0023.
D.F. and V.B. partially supported by SupToPhAG/EIPHI ANR-17-EURE-0002, Région Bourgogne-Franche-Comté, Feder Bourgogne and Bridges ANR-21-CE40-0017.

We warmly thank Christian Pauly, Sasha Kuznetsov and Jerzy Weyman for useful 
discussions. Special thanks also to Shigeru Mukai and Akihiro Kanemitsu for sharing the results of \cite{km}. 
\end{ack}

\section{Lines in the Coble quartic} 

Throughout the text we will denote by $\mathrm{U}_C(r,d)$ the moduli
space of semi-stable vector bundles on a curve $C$ of rank $r$ and determinant of degree $d$. If $L$ is a degree $d$ line bundle on $C$, we will denote by $\SU_C(r,L)$ the subvariety of $U_C(r,d)$ parametrizing vector bundles of determinant $L$; moreover $\SU_C(r):=\SU_C(r,\cO_C)$. Since all the moduli spaces $\SU(r,L)$ are (non canonically) isomorphic when the degree of $L$ is fixed, we will also denote their isomorphism class by $\SU_C(r,d)$; it does depend on $d$ only modulo $r$. Finally,  we will denote by $\mathrm{U}_C(r,d)^{\mathrm{eff}}$ the moduli space of vector bundles with effective determinant. When $d=1$, this  moduli space fibers over the curve $C$ with fiber over $c$ isomorphic to $\SU_C(2,\cO_C(c))$.

\subsection{Covering families of rational curves in $\SU_C(2)$}
Rational curves in the moduli spaces $\SU_C(r,d)$ were extensively studied, see e.g. 
\cite{nr75, opp, hwang, hr, sun, mok-sun, pal}.
Restricting to $g=3$,  $r=2$ and $d=0$, the results of 
\cite{mustopa-teixidor} show that there exist two different families of covering 
lines i.e., families of rational curves of degree one with respect to the Theta embedding
$$\cC:=\SU_2(C)\hookrightarrow |2\Theta|=\PP(V_8),$$
passing through a general point of the moduli space. We will denote these two families by $\cF_H$ and $\cF_R$ and consider them as subvarieties of the Grassmannain $G(2,V_8)$. In the sequel 
we describe these two covering families in some detail. They are both of dimension 
six but behave very differently; we will illustrate this by showing how different  are the corresponding VMRT's (variety of minimal rational tangents), which in our case,
since we deal with lines, are just the spaces of lines through a fixed general point. 

\subsection{Hecke lines}\label{sec_hecke_lines}
A generic Hecke line can be described by choosing a point $c\in C$, and a rank two vector 
bundle $F$ on $C$ with determinant $\det(F)=\cO_C(c)$. Then the 
bundles  $E$ that fit into an exact sequence 
$$0\lra E\lra F\lra \cO_c\lra 0$$
are  parametrized by $\PP (F_c^\vee)\simeq\PP^1$. They have trivial determinant
and are all stable when $F$ is $(1,0)$-semistable in the sense of \cite[Definition 2.5]{mok-sun}. For vector bundles of rank two and degree one, this condition 
is equivalent to stability, hence also to semistability. 
The resulting curve in $\SU_C(2)$ is a line and
such lines are called Hecke lines. Note that dualizing, we get an exact sequence
$$0\lra F^\vee\lra E^\vee\lra \cO_c\lra 0,$$
so a Hecke line parametrizes all the possible extensions 
of $\cO_c$ by $F^\vee$. 

By \cite[Remark 5.3]{pal},
a general Hecke line defines a vector bundle $\cE$ of rank $2$ over $C\times \PP^1$ 
fitting into an exact sequence 
$$0\lra p_1^*F^\vee\otimes p_2^*\cO_{\PP^1}(-1)\lra \cE^\vee\lra p_1^*\cO_c\lra 0,$$
where $p_1$ and $p_2$ are the projections of $C \times \PP^1$ onto the two factors $C$ and $\PP^1$.
An easy consequence is that, since $\cE^\vee$ admits a \emph{unique} jumping line at $c$,
this point can be uniquely recovered from the Hecke line. (Beware this is only 
true for general Hecke lines.)  

We will denote by $\cF_H$ the family of Hecke lines in $\SU_C(2,\cO_C)$, considered
as a subvariety of the space $G(2,V_8)$ of lines in $\PP(V_8)$. 

\begin{remark}\label{hecke-fibers}
Although a Hecke line does not always define a unique point in $C$, once we have
fixed such a point $c$ there is a well-defined morphism from $\SU_C(2,\cO_C(c))$
to $\cF_H$. By the previous observations, the resulting morphism from $\tilde\cF_H:=\mathrm{U}_C(2,1)^{\mathrm{eff}}$ to $\cF_H$ is birational.
\end{remark}

Conversely, Hecke lines passing through a general point $[E]$ of $\SU_C(2,K_C)$
(we make this choice of determinant just for convenience) 
are obtained by choosing a projection $E\ra E_c\ra\cO_c$, where $E_c$ denotes the 
fiber of the vector bundle $E$ at the point $c\in C$. So  they are parametrized
by (the image in $\cF_H$ of) the total space of the projective bundle $\PP(E^\vee)$ 
over $C$. The tangent map of this morphism sends $\PP(E^\vee)$ to the tangent space 
of the moduli space at $[E]$, which is the projectivization of 
$$H^1(C,\mathcal{E}nd_0(E))\simeq H^0(C,K_C\otimes \mathcal{E}nd_0(E))^\vee\simeq H^0(C,S^2E)^\vee,$$
since $K_C\simeq  \det(E)$. Here $ \mathcal{E}nd_0(E)$ denotes the vector bundle of 
traceless endomorphims of $E$. This implies (see \cite{hwang, hr} for more general statements):

\begin{prop}
The VMRT of the family $\cF_H$  of Hecke lines at a general point $[E]$ of the moduli space
is the image of the ruled surface $\PP(E^\vee)$ by the linear system $|\cO_E(2)|$. 
In particular this surface contains no line.
\end{prop}

Equivalently, the latter claim means that a general Hecke line is not contained 
in any larger linear space contained in $\SU_C(2)$, although such larger linear spaces 
do exist.

\subsection{Lines in the ruling}\label{sec_ruling_lines}

For each line bundle $L\in \Pic^1(C)$, consider the rank two vector bundles $E$ obtained
as extensions of the form
$$0\lra L\lra E\lra K_C\otimes L^\vee\lra 0.$$
Such extensions are parametrized by $\PP_L:=\PP (\Ext^1(K_C\otimes L^\vee,L))\simeq\PP^3$. 
Hence a ruling of  $\SU_C(2)$ by a family of $\PP^3$'s parametrized by $\Pic^1(C)$, 
which we denote by $\PP(\cR)\ra \Pic^1(C)$.

Note that $\PP_L$ intersects the Kummer threefold along a copy $C_L$ of $C$
\cite[1.1]{opp}.  
According to \cite[Theorem 1.3]{opp}, a line in $\PP_L$ is a Hecke line if and 
only if it meets $C_L$. 

Moreover, by \cite[Proposition 1.2]{opp}, two spaces $\PP_L$ and $\PP_M$ are always distinct for $L\ne M$ and, for sufficiently general choices of $L$ and $M$, they are disjoint. When they meet, their intersection is a single point, or a line; the latter case happens exactly when $K_C-L-M$ is effective. 
In particular, if a line is contained in $\PP_L\cap\PP_M$, it must be a bisecant 
to both $C_L$ and $C_M$. 

\smallskip 
Now consider the family $\cF_R$ of lines contained in the $\PP^3$'s of the ruling.
By what we have just recalled, $\cF_R$ is the birational image in $G(2,V_8)$ of 
the quadric bundle 
$G(2,\cR)$ over $\Pic^1(C)$.

\begin{prop}\label{vmrt-ruling}
The VMRT at a general point of $\SU_C(2)$, of the family $\cF_R$ of lines in its ruling,
is the disjoint union of eight planes in $\PP^5$. 
\end{prop}

\proof It follows from \cite[Section 4.1]{pauly} that the map $\PP(\cR)\lra \SU_C(2)$ is generically 
finite of degree $8$. This means that eight $\PP^3$'s of the ruling
pass through a general point $[E]$ of $\SU_C(2)$, and for each of them the 
lines passing through $[E]$
are parametrized by a projective plane. Finally, these projective planes
are disjoint, again by \cite[Proposition 1.2]{opp}.\qed 

\medskip For future use we record the following easy consequence.

\begin{coro}
Any plane in $\SU_C(2,K_C)$ passing through a general point is contained in a unique $\PP^3$ of the ruling. 
\end{coro}

\section{Four-forms and orbital degeneracy loci}

In this section we recall the definitions of some orbital degeneracy loci closely connected to the geometry of $\SU_C(2,\cO_C)$, for $C$ a general curve of genus $3$. In particular we recall how to recover the Coble quartic from a general four-form in eight variables. Using this description, we give a short proof of the self-duality statement of \cite{pauly}.
Our references for orbital degeneracy loci (sometimes abbreviated as ODL) are 
\cite{bfmt, bfmt2}.

\smallskip\noindent {\it Notation.}
We will denote by $V_n$ and $U_i$ complex vector spaces of dimension $n$ and $i$, respectively (usually $V_n$ will be fixed and $U_i$ will be a variable subspace of $V_n$). We will also denote by $G(i,V_n)$ the Grassmannian of $i$-dimensional subspaces of $V_n$ and by $Fl(i_1,\ldots,i_k,V_n)$ the flag variety of flags of subspaces of $V_n$ of dimensions $i_1<\cdots <i_k$. Over the flag variety, we will denote by $\cU_{i_j}$ the rank-$i_j$ tautological bundle; over the Grassmannian we will denote by $\cU$ the tautological bundle and by $\cQ$ the quotient tautological bundle.

\subsection{A simple construction of the Coble quartic}
In this section we recall some results from \cite{GSW}. 
The starting point is a general four-form in eight variables, $v\in\wedge^4V_8\simeq 
\wedge^4V_8^\vee$, where $V_8$ denotes a complex eight-dimensional vector space.
Recall that this is a {\it theta-representation}, being part of a $\ZZ_2$-grading 
of the exceptional Lie algebra 
$$\fe_7\simeq \fsl(V_8)\oplus \wedge^4V_8.$$
The action of the so-called theta-group, which here is $\SL(V_8)$, behaves very 
much as the action of the adjoint group on a simple complex Lie algebra. In particular 
one has Jordan decompositions, and the GIT-quotient 
$$\wedge^4V_8 /\hspace{-1mm}/ \SL(V_8) \simeq \fh/W$$
for some finite complex reflection group $W$ acting on what is called a  Cartan subspace $\fh$
of the theta-representation. We will make this Cartan subspace explicit later on. For now 
we just need to know that it coincides with  the seven-dimensional representation 
of the Weyl group of $E_7$. As a consequence, the choice of $v$ determines uniquely a 
non-hyperelliptic curve 
$C$ of genus three (a plane quartic) with a marked flex point \cite[Remark 6.1]{GSW}. 

\smallskip
We will construct from our general $v\in\wedge^4V_8$ a collection of geometric objects
defined as orbital degeneracy loci. The main point of this approach is that it allows to reduce to simpler representations. Typically, the Borel-Weil theorem gives an isomorphism  
$$\wedge^4V_8\simeq H^0(\PP(V_8),\wedge^4\cQ) \simeq H^0(\PP(V_8),\wedge^3\cQ^\vee(1)),$$
where $\cQ$ denotes the rank seven quotient vector bundle on $\PP(V_8)$. At the price 
of passing to a relative setting over $\PP(V_8)$, this reduces the study of $\wedge^4V_8$
to that of three-forms in seven variables. 

But then the situation is much simpler, because if $V_7$ is a seven-dimensional complex vector space, 
$\wedge^3V_7^\vee\cong \wedge^4 V_7$ has finitely many orbits under the action of $\GL(V_7)$. Each orbit 
closure $Y$ allows to associate to $v\in \wedge^4V_8$ the locus $D_Y(v)\subset \PP(V_8)$
of points $x$ where the image of $v$ lies in the corresponding $Y_x\subset \wedge^3\cQ^\vee(1)_x$ (this is exactly how orbital degeneracy loci are defined). 
By the general results of \cite{bfmt}, for $v$ general the main 
properties of $Y$ will be transferred to $D_Y(v)$, starting from its codimension. We can therefore focus on the orbit closures in  $\wedge^3V_7^\vee$ of  codimension at most seven. 
Remarkably, there are only three such orbit closures (not counting the whole space), that 
we can index by their codimension: $Y_1$ is a hypersurface of degree $7$, $Y_4$ is its 
singular locus,  $Y_7$ is the singular locus of $Y_4$. The corresponding orbital degeneracy
loci have been described in \cite[6.1, 6.2]{GSW}.

\begin{prop}
For $v$ general, the threefold $\cK:=D_{Y_4}(v)$ is the Kummer variety of a non-hyperelliptic genus three curve $C$. It is the singular locus of the quartic hypersurface $\cC:=D_{Y_1}(v)$. Its singular locus  is the finite set 
$\cK[2]:=D_{Y_7}(v)$.
\end{prop}

Since the Coble quartic can be characterized as the unique quartic hypersurface that is singular along the Kummer threefold \cite[Proposition 3.1]{beauville}, we can immediately 
deduce that it coincides with $D_{Y_1}(v)$.

\subsection{Kempf collapsings} 

A nice feature of our orbital degeneracy loci is the following. It turns out that 
the orbit closures they are associated to, although singular, admit nice resolutions
of singularities by {\it Kempf collapsings}, which are birational contractions from
total spaces of homogeneous vector bundles on flag manifolds. These homogeneous 
 vector bundles are typically non-semisimple, making them more difficult to handle. 
 Nevertheless, these collapsings allow to describe the corresponding orbital degeneracy 
 loci in terms of zero loci of sections of vector bundles. 
 
 In the cases we are  interested in, we obtain the following descriptions, where $U_k$ stands for a $k$-dimensional subspace of $V_8$. For $A,B$ subspaces of a vector space $V$, we will denote by $(\wedge^p A) \wedge (\wedge^q B) \subset \wedge^{p+q}V$ the linear subspace spanned by the elements of the form $a_1\wedge \cdots \wedge a_p 
 \wedge b_1 \wedge \cdots \wedge b_q$ with $a_1,\ldots,a_p \in A$ and 
 $b_1,\ldots,b_q\in B$. 
For vector subbundles $\mathcal{A},\mathcal{B}$ of a the trivial bundle $V \otimes \cO$, we use the same convention to define $(\wedge^p \mathcal{A}) \wedge (\wedge^q \mathcal{B})$ in $\wedge^{p+q} V \otimes \cO$.

 \begin{prop}
 The Coble quartic $\cC$ can be described as 
\[
\Big\{ [U_1]\in \PP(V_8) \mid \exists U_4\supset U_1, \;
v\in (\wedge^2 U_4) \wedge (\wedge^2 V_8) +\wedge^3 V_8\wedge  U_1\Big\} .
\]
The Kummer threefold $\cK$ is
\[
\Big\{ [U_1]\in \PP(V_8) \mid \exists U_6\supset U_2\supset U_1, \;
v\in \wedge^4 U_6+ \wedge^2 U_6\wedge U_2\wedge V_8  + \wedge^3 V_8\wedge  U_1\Big\} .
\]
The singular locus $\cK[2]$ of $\cK$ is 
\[
\Big\{ [U_1]\in \PP(V_8) \mid \exists U_7\supset U_4\supset U_1, \;
v\in \wedge^3 U_4\wedge V_8 + (\wedge^2 U_4)\wedge (\wedge^2 U_7) +\wedge^3 V_8\wedge  U_1\Big\} .
\]
\end{prop}
 
These results follow from a combination of \cite[Section 6]{GSW} and \cite[Section 3]{KWE7}. Let us clarify the statement for instance for $\cK$, the explanations for  the other loci being similar. In \cite[Section 6]{GSW} it is shown that $\cK:=D_{Y_4}(v)$ is the Kummer threefold. In \cite[Section 3]{KWE7} it is proved that $Y_4\subset \wedge^4 V_7 $ is desingularized by a the total space of the vector bundle 
$$\cW:=\wedge^4 \cU_5+ \wedge^2 \cU_5\wedge \cU_1 \wedge V_7$$ 
over the flag variety $Fl(1,5,V_7)$. Here we denoted by $\cU_1$ and $\cU_5$ respectively the rank one and rank five tautological vector bundles on $Fl(1,5,V_7)$. The projection from the total space of $\cW$ to $Y_4\subset \wedge^4 V_7$ is given by the composition of the inclusion of $\cW$ inside $\wedge^4 V_7 \otimes \cO_{Fl(1,5,V_7)}$ with the  projection
to $\wedge^4 V_7$. For $v$ general, this desingularization $\Tot(\cW)\to Y_4$ of $Y_4$ 
can be \emph{relativized} to obtain a desingularization of $D_{Y_4}(v)$, as explained in \cite[Section 2]{bfmt}. For this we simply consider the flag bundle $Fl(1,5,\cQ)$: by the previous discussion, any point of $x=[U_1]\in D_{Y_4}(v)$ must be the image of a flag $\bar{U}_1
\subset \bar{U}_5\subset \cQ_x=V_8/U_1$ such that $v \; \mathrm{mod} \; U_1$ belongs to 
$\wedge^4 \bar{U}_5+ \wedge^2 \bar{U}_5\wedge \bar{U}_1 \wedge \cQ_x\subset \wedge^4\cQ_x$. 
This flag originates from a flag $(U_1\subset U_2\subset U_6\subset V_8)$
(such that $\bar{U}_1=U_2/U_1$, etc.),
and we can rewrite the previous condition as asking that $x=[U_1]$ belongs to
the projection of 
$$Z(v):=\Big\{ (U_1\subset U_2\subset U_6)\in Fl(1,2,6,V_8), \;
v\in \wedge^4 U_6+ \wedge^2 U_6\wedge U_2\wedge V_8  + \wedge^3 V_8\wedge  U_1\Big\}.$$ This is the zero locus of a global section of a globally generated bundle, obtained 
as a quotient of the trivial bundle with fiber $\wedge^4V_8$. For $v$ general this
section is  general,  so $Z(v)$ is smooth. Moreover the projection $Z(v)\to D_{Y_4}(v)\subset \PP(V_8)$, obtained by just forgetting $U_2$ and $U_6$, is birational.

\subsection{Self-duality of the Coble quartic}
Because of the natural isomorphism $\wedge^4V_8\simeq \wedge^4V_8^\vee$ (defined up to scalar, or more precisely up to the choice of a volume form on $V_8$), the same constructions can be performed
in the dual projective space $\PP(V_8^\vee)$. This is related to the remarkable 
fact that the Coble quartic is projectively self-dual \cite{pauly}. Let us show
how this duality statement easily follows from our approach in terms of orbital 
degeneracy loci. 

First consider a general point $[U_1]$ of $\cC=D_{Y_1}(v)$. As we have seen in the previous section, there exists (a unique) $U_4\supset U_1$ such that $v$ belongs to $(\wedge^2 U_4) \wedge (\wedge^2 V_8) +\wedge^3 V_8\wedge  U_1$. Reducing modulo $(\wedge^2 U_4) \wedge (\wedge^2 V_8)$, we get 
$$\bar{v}\in \wedge^3(V_8/U_4)\otimes U_1\simeq (V_8/U_4)^\vee.$$
In general $\bar{v}$ is nonzero and defines a hyperplane in $V_8/U_4$, that is, 
a hyperplane $U_7$ of $V_8$, containing $U_4$. Note that this exactly means that 
\begin{equation}
    \label{key}
v\in (\wedge^2 U_4) \wedge (\wedge^2 V_8) +\wedge^3 U_7\wedge  U_1.
\end{equation}

\begin{lemma}
$\PP(U_7)$ is the tangent hyperplane to $\cC$ at $[U_1]$.
\end{lemma}

\proof Let $\tilde{\cC}$ denote the variety of flags $(U_1\subset U_4)$ such that $v$
belongs to  $$\cF(U_1,U_4):=(\wedge^2 U_4) \wedge (\wedge^2 V_8) +\wedge^3 V_8\wedge  
U_1.$$
We know that the projection $\tilde{\cC}\lra\cC$ is birational. Moreover, as a 
subvariety of the flag manifold $Fl(1,4,V_8)$, $\tilde{\cC}$ is the zero-locus of the section of the vector bundle $\wedge^4V_8/\cF$ defined by $v$. Let $\fp(U_1,U_4)$ denote
the stabilizer of the flag $(U_1\subset U_4)$ inside $\fgl(V_8)$. The tangent space to 
$Fl(1,4,V_8)$ at the corresponding point is the quotient  $\fgl(V_8)/\fp(U_1,U_4)$;
and the tangent space to $\tilde{\cC}$ is the image, in this quotient, of the space
of endomorphisms $X\in\fgl(V_8)$ such that $X(v)$ belongs to $\cF(U_1,U_4)$, as 
follows from the normal exact sequence. The tangent space to $\cC$ is then the image of this space inside $\fgl(V_8)/\fp(U_1)\simeq \Hom(U_1,V_8/U_1)$, where $\fp(U_1)$ denotes the stabilizer of the line $U_1$. 

So our claim will follow, if we can check that any $X\in\fgl(V_8)$ such that $X(v)$ 
belongs to $\cF(U_1,U_4)$, must send $U_1$ into the hyperplane $U_7$. But  (\ref{key})
implies, once we apply $X$, that
$$X(v)\in U_4 \wedge (\wedge^3 V_8) +\wedge^3 U_7\wedge  X(U_1).$$
If $X(v)$ belongs to $\cF(U_1,U_4)$, it has to vanish modulo $U_4$. So 
$\wedge^3 U_7\wedge  X(U_1)$ must also vanish modulo $U_4$, which is the case
only if $X(U_1)\subset U_7$.\qed

\medskip 
Recall that once we fix a volume form on $V_8$, 
we get an isomorphism of $\wedge^4V_8$ with 
 $\wedge^4V_8^\vee$. We will denote by $v^\vee$ the image of $v$. (Strictly speaking 
 it is uniquely defined only up to scalar, but this is irrelevant in our constructions.)
To make things clearer we will denote by $\cC(v)$ the Coble quartic defined by $v$
in $\PP(V_8)$, and by $\cC(v^\vee)$ the Coble quartic defined by $v^\vee$
in $\PP(V_8^\vee)$.

\begin{theorem}\label{selfdual}
The projective dual of $\cC(v)$ is  $\cC(v^\vee)$.
\end{theorem}

\proof For $[U_1]$ a general point of $\cC$, we have a flag $(U_1\subset U_4\subset U_7)$
such that $v$ belongs to $(\wedge^2 U_4) \wedge (\wedge^2 V_8) +\wedge^3 U_7\wedge  U_1$.
Choose an adapted basis $e_1,\ldots , e_8$, so that $e_1$ generates $U_1$, etc. 
The condition means that $v$ is a linear combination of elementary tensors $e_i\wedge e_j\wedge e_k\wedge e_\ell$ with $i,j\le 4$, and of $e_5\wedge e_6\wedge e_7\wedge e_1$. 

Now recall that if the chosen volume form on $V_8$ is $e_1\wedge \cdots\wedge e_8$, and 
$e_1^\vee,\ldots , e_8^\vee$ is the dual basis of $e_1,\ldots , e_8$, then the isomorphism of
 $\wedge^4V_8$ with $\wedge^4V_8^\vee$ sends the elementary tensor $e_i\wedge e_j\wedge e_k\wedge e_\ell$ to  $\pm e_p^\vee\wedge e_q^\vee\wedge e_r^\vee\wedge e_s^\vee$, where $\{i,j,k,l\}
 \cap \{p,q,r,s\}=\emptyset$.
 
 As a consequence, $v^\vee$ will be a linear combination of elementary tensors $e_p^\vee\wedge e_q^\vee\wedge e_r^\vee\wedge e_s^\vee$ with $p,q\ge 5$, and $e_2^\vee\wedge e_3^\vee\wedge e_4^\vee\wedge e_8^\vee$. In other words,  
 $$v^\vee \in (\wedge^2 U_4^\perp) \wedge (\wedge^2 V_8^\vee) +\wedge^3 U_1^\perp\wedge  U_7^\perp.$$
 This is exactly the condition that ensures that $[U_7^\perp]$ belongs to $\cC(v^\vee)$.
 Thanks  to the previous Lemma we deduce that $\cC(v)^\vee\subset \cC(v^\vee)$. Moreover, 
 the symmetry between $U_1$ and $U_7^\perp$ implies that in general, $U_1$ 
 can be recovered 
 from $U_7$ exactly as $U_7$ is constructed from $U_1$, which means that $\cC(v)^\vee$
 is birationally equivalent to $\cC(v)$. Finally, 
 since $\cC(v)^\vee$ and   $\cC(v^\vee)$  are both irreducible hypersurfaces, they must be equal.\qed

\medskip The previous discussion shows that it is natural to define the variety 
$\cC(v,v^\vee)\subset Fl(1,4,7,V_8)$ parametrizing the flags $(U_1\subset U_4\subset U_7\subset V_8) $ satisfying condition (\ref{key}). This is a smooth variety 
dominating birationally both $\cC(v)$ and $\cC(v^\vee)$; there is a diagram
$$\xymatrix@-2ex{
&& Fl(1,4,7,V_8) &&\\
Fl(1,4,V_8) & & \cC(v,v^\vee)\ar[dr]\ar[dl]\ar@{^{(}->}[u] && Fl(4,7,V_8)\\
 & \tilde{\cC}(v)\ar[dl]\ar@{_{(}->}[ul] & & \tilde{\cC}(v^\vee)\ar[dr]\ar@{^{(}->}[ur] &\\
 \PP(V_8)\supset \cC(v)  \ar@{-->}[rrrr]^{d\cC}& & &&\cC(v^\vee)\subset\PP(V_8^\vee) 
}$$
One recovers that way the constructions explained is \cite[section 3.3]{pauly}.
We used the suggestive notation $d\cC$ for the Gauss map, which sends a smooth point 
of $\cC(v)$ to its tangent hyperplane, given by the differential of the cubic's equation.

\subsection{The Cartan subspace} 
Recall that a Cartan subspace for the $\ZZ_2$-graded Lie algebra $\fe_7=\fsl(V_8)\oplus  \wedge^4V_8$ is a maximal subspace of $\wedge^4V_8$, made of elements of $\fe_6$ which
are semisimple and commute \cite{vinberg}. Among other nice properties, a general element of  $\wedge^4V_8$ is $\SL(V_8)$-conjugate to (finitely many) elements of any given 
Cartan subspace.

An explicit Cartan subspace of $\wedge^4V_8$ is worked out in \cite[(3.1)]{oeding}. It coincides with the space of Heisenberg invariants provided in \cite[Remark 4.2]{rsss}. Here is a list of seven generators, for a given basis $e_1,\ldots , e_8$ of $V_8$:
$$\begin{array}{rcl}
h_1 & = & e_1\wedge e_2\wedge e_3\wedge e_4 +e_5\wedge e_6\wedge e_7\wedge e_8 ,\\
h_2 & = & e_1\wedge e_3\wedge e_5\wedge e_7 +e_6\wedge e_8\wedge e_2\wedge e_4 ,\\
h_3 & = & e_1\wedge e_5\wedge e_6\wedge e_2 +e_8\wedge e_4\wedge e_3\wedge e_7 ,\\
h_4 & = & e_1\wedge e_6\wedge e_8\wedge e_3 +e_4\wedge e_5\wedge e_7\wedge e_2 ,\\
h_5 & = & e_1\wedge e_8\wedge e_4\wedge e_5 +e_7\wedge e_2\wedge e_6\wedge e_3 ,\\
h_6 & = & e_1\wedge e_4\wedge e_7\wedge e_6 +e_2\wedge e_3\wedge e_8\wedge e_5 ,\\
h_7 & = & e_1\wedge e_7\wedge e_2\wedge e_8 +e_3\wedge e_5\wedge e_4\wedge e_6.
\end{array}$$
Combinatorially, each of these generators is given by a pair of complementary fourtuples of indices in $\{1,\ldots,8\}$. Each of these $14$ fourtuples shares a pair of indices with any other distinct, not complementary fourtuple. This is the property that ensures the commutation in $\fe_6$, since the Lie bracket of $\fe_6$, restricted to $\wedge^4V_8$, is given by the unique (up to scalar) $\fsl_8$-equivariant morphism 
$$\wedge^2(\wedge^4V_8) \lra \wedge^4V_8\otimes \wedge^4V_8\lra S_{21111110}V_8\simeq\fsl_8.$$
If we start with two elementary tensors given by fourtuples with a common pair of indices, we can include them into $\wedge^4U_6$ for some codimension two susbpace $U_6\subset V_8$. But then the Lie bracket factors through $S_{21111110}U_6=\{0\}$, so it has to 
vanish. 

\smallskip
Each pair of indices in $\{1,\ldots,8\}$ 
belongs to three of the $14$ fourtuples. For any triple 
$(ijk)$ among $$(124), \; (137),\;  (156),\;  (235),\; (267),\; (346), \; (457),$$
$h_i$, $h_j$ and $h_k$ share four disjoint pairs (for example  $h_1$, $h_2$ and 
$h_4$ share $(13), (24), (57), (68)$). These seven triples always meet in exactly 
one index, so they are in correspondence with the lines in a Fano plane. More on this in \cite[Section 4]{config}. 

\smallskip
A nice consequence of this description is the following 

\begin{prop}\label{iso}
$\cC(v)$ and $\cC(v^\vee)$ are isomorphic.
\end{prop}

\proof Since our $v$ is general, we may suppose up to the action of $\SL(V_8)$ that 
$v$ belongs to our Cartan subspace above, given in terms of the basis $e_1, \ldots, e_8$
of $V_8$. 
Denote the dual basis by $e_1^\vee, \ldots , e_8^\vee$, and choose the volume form $e_1^\vee\wedge \cdots \wedge e_8^\vee$ on $V_8$. Then the induced isomorphism from 
$\wedge^4V_8$ to $\wedge^4V^\vee_8$ sends $e_I=e_{i_1}\wedge e_{i_2}\wedge e_{i_3}
\wedge e_{i_4}$ to $\epsilon_{I,J}e_J$, where $J$ is the complement of $I$ in $\{1,\ldots , 8\}$ and $\epsilon_{I,J}$ is the sign of the permutation $(i_1,\ldots , i_4,j_1,\ldots, j_4)$. 

Now, observe that for each $i$, $h_i$ is of the form $e_K+e_L$ for two complementary 
sets of indices $K$ and $L$. Moreover, one can check that $\epsilon_{K,L}$ is always equal to $1$. This implies that $h_i^\vee=e^\vee_K+e^\vee_L$ has exactly the same expression as $h_i$, in terms of the dual basis. In other words the map $v\mapsto v^\vee$,
when restricted to our Cartan subspace, is essentially the identity,  and the claim
follows. \qed

\subsection{The abelian threefold} 

Remarkably, one can construct the abelian threefold whose Kummer variety is $\cK$
by considering another orbital degeneracy locus. The idea is to use the flag variety
$Fl(1,7,V_8)$, the incidence correspondence in $\PP(V_8)\times \PP(V_8^\vee)$ parametrizing flags $(U_1\subset U_7)$. 
The rank six quotient bundle $\cN=\cU_7/\cU_1$ allows to realize the space 
of four-forms as 
$$\wedge^4V_8^\vee=H^0(Fl(1,7,V_8), p_1^*\cO(1)\otimes\wedge^3\cN^\vee).$$
Exactly as before, this allows to associate to any $\GL(V_6)$-orbit closure $Y$ in $\wedge^3V_6$ an 
orbital degeneracy locus $D_Y(v)\subset Fl(1,7,V_8)$. Here $V_6$ is a six-dimensional vector space. In particular, the cone $Y_{10}$ over 
the Grassmannian $G(3,V_6)$ yields, for $v$ generic, a smooth threefold $\cA:=D_{Y_{10}}(v)$.  
In similar terms as for the other orbital degeneracy loci,  
this threefold is 
\begin{equation}\label{defA}
\cA=\Big\{ [U_1\subset U_7]\in Fl(1,7,V_8) \mid \exists U_1\subset U_4\subset U_7,\;
v\in \wedge^3 U_4\wedge V_8 +\wedge^4 U_7 +\wedge^3 V_8\wedge  U_1\Big\} .
\end{equation}

\begin{prop}
$\cA$ is a torsor over an abelian threefold, and the projection to $\PP(V_8)$ is a double cover of $\cK$. 
\end{prop}

\proof This is \cite[Proposition 6.12]{GSW}. \qed

\medskip Over a point of $\cA$, given by a flag $U_1\subset U_7$, the four-form 
$v$ defines a decomposable tensor in $\wedge^3(U_7/U_1)$. This tensor is never
zero if $v$ is general, and therefore defines a four-dimensional space $U_4$ 
such that $U_1\subset U_4\subset U_7$. Hence a rank-four vector bundle $\cU_4$ 
on $\cA$, a subbundle of the trivial bundle $V_8\otimes \cO_\cA$.

\begin{remark}\label{8fold} The proper orbit closures of the $\GL(V_6)$-action on $\wedge^3V_6$ are,
apart from the cone over the Grassmannian, a quartic hypersurface and the codimension five locus
of partially decomposable tensors. In our relative setting, the quartic induces a hypersurface of
bidegree $(2,2)$, whose singular locus is an eightfold that is singular exactly along $\cA$. 
So once again we get a very interesting singular hypersurface. It would be very nice to find a modular interpretation of these loci. \end{remark}

\section{Lines from alternating forms}

In this section we will identify the two covering families of lines in $\SU_C(2)$ in terms of orbital degeneracy loci; this will give a very explicit description of these families in terms of existence of special flags of vector spaces. As a consequence of this we will obtain Theorem \ref{thm_main}, in which we identify the moduli space $\SU_C(2,L)$, for $L$ of odd degree, with an orbital degeneracy locus in $G(2,V_8)$ associated to $v\in \wedge^4 V_8$.

\subsection{The ruling and its lines}
Recall the definition of the abelian threefold $\cA$ from Equation \eqref{defA}. 
Our next result relates it to the ruling described in section \ref{sec_ruling_lines}.

\begin{prop}
The family $\PP(\cU_4)$ over $\cA$ coincides with the ruling $\PP(\cR)$ over 
$\Pic^1(C)$ of the moduli space $\SU_C(2)$.
\end{prop}

\proof We need to prove that for any flag $(U_1\subset U_7)$ in $\cA$, defining the four-plane $U_4$, the linear space $\PP(U_4)$ is contained in $\cC$. If we can 
show that $\cC$ is even covered by this family of $\PP^3$'s, we will be done since 
the ruling is unique. 
So let us prove these two statements. 

\begin{lemma}
The image of $\PP(\cU_4)$ in $\PP(V_8)$ is contained in 
the Coble quartic.
\end{lemma}

\proof 
Consider a point of $\cA$ and the associated flag $U_1\subset U_4\subset U_7$.
By the very definition of $\cA$, this means we can write 
$$v=e_1\wedge w+v'+e_2\wedge e_3\wedge e_4\wedge e_8$$
for some vectors $e_1\in U_1$ and  $e_2,e_3,e_4\in U_4$, with $w\in \wedge^3V_8$ 
and $v'\in\wedge^4U_7$. Under the generality hypothesis we can suppose that 
$U_4=\langle e_1,e_2,e_3,e_4\rangle$, and it suffices to check that $U'_1=\CC e_2$
defines a point of $\cC$. 

Modulo $e_1$ and $e_2$, the tensor $w$ if a three-form in six variables. Since the 
secant of the Grassmannian $G(3,6)$ in its Pl\"ucker embedding fills in the whole ambient projective space, generically we can write $w=a\wedge b\wedge c+d\wedge e\wedge f$
modulo $e_1$ and $e_2$, for some vectors $a,b,c,d,e,f$. Modulo $e_1$ and $e_2$ again,
$v'$ is a four-form in only five variables, so it defines a hyperplane that will cut 
the three-dimensional space $\langle a,b,c\rangle$ in codimension one, say along 
 $\langle a,b\rangle$, and similarly it will cut $\langle d,e,f\rangle$ in codimension one, say along $\langle d,e\rangle$. In other words, we may suppose that  modulo $e_1$ and $e_2$, $v'=a\wedge b\wedge d\wedge e$. But then, modulo $e_2$ we get 
 $$v=e_1\wedge (a\wedge b\wedge c+d\wedge e\wedge f)+a\wedge b\wedge d\wedge e.$$
 So $v$ belongs to $(\wedge^2U'_4)\wedge (\wedge^2V_8)+\wedge^3V_8\wedge U'_1$ if $U'_4=\langle e_1,e_2,a,d\rangle.$ The existence of such a space $U'_4\supset U'_1$ is precisely the 
 required condition for $U'_1$ to belong to $\cC$, so we are done. \qed

\begin{lemma}\label{U4covers}
The family $\PP(\cU_4)$ covers the Coble quartic.
\end{lemma}

\proof  This can be done by a Chern class computation, being equivalent to the 
fact that the degree of $\PP(\cU_4)$ with respect to the relative hyperplane
class does not vanish. Notice that by Equation \eqref{defA}, $\cA$ can be considered as a subvariety of $Fl(1,4,7,V_8)$. Even more, it is the zero locus in the flag manifold of the section $\overline{v}$ of the rank $19$ vector  bundle 
$$\cG:=\wedge^4 V_8 /(\wedge^3 \cU_4\wedge V_8 + \wedge^4 \cU_7 + \wedge^3 V_8\wedge \cU_1)$$ over $Fl(1,4,7,V_8)$ defined by $v$. Since this section is general, the class of $\cA$ in the Chow ring of the flag manifold is the top Chern class of $\cG$. So the degree we are looking for is 
$$\int_{\PP(\cU_4)}c_1(\cU_1^\vee)^6=\int_{\cA}s_3(\cU_4^\vee)=\int_{Fl(1,4,7,V_8)}c_{19}(\cG)s_3(\cU_4^\vee)=32,$$
as can be computed using \cite[Schubert2 package]{Macaulay2}. This implies the claim.
\qed

\medskip\noindent {\it Remark.} $32$ is the expected number: since the Coble hypersurface 
has degree $4$, we recover the fact that exactly $8$ $\PP^3$'s of the ruling pass through
a general point of the quartic, as recalled in the proof of Proposition \ref{vmrt-ruling}.

\medskip The previous statement allows to reconstruct the curve $C$ purely in terms 
of the four-form and its associated orbital degeneracy loci. Indeed, we have recalled
that a $\PP^3$ of the ruling meets the Kummer threefold along a copy of the curve. 

\begin{coro}
For any point of $\cA$, with associated flag $(U_1\subset U_4\subset U_7)$, the intersection of $\PP (U_4)$ with $\cK$ is a copy of the curve $C$. 
\end{coro}

And of course we also recover the family of lines in the ruling as a quadric bundle. 
Indeed, the same arguments as in section 2.3 yield:

\begin{coro}
The total space of the fiber bundle $G(2,\cU_4)$ over $\cA$ maps 
birationally to the family $\cF_R$ in $G(2,V_8)$.
\end{coro} 

\subsection{Hecke lines from alternating forms}
\label{sec_hecke}
In the previous section we have defined some ODL $D_{Y_i}(v)$ from orbits inside the space of three-forms in seven variables (i.e., in the notation of the previous sections, inside $\wedge^3 V_7$). We will use a similar construction to obtain ODL inside the Grassmannian $G(2,V_8)$. The Borel-Weil theorem gives an isomorphism  
$$\wedge^4V_8\simeq H^0(G(2,V_8),\wedge^4\cQ)=H^0(G(2,V_8),\wedge^2\cQ^\vee(1)),$$
where $\cQ$ denotes the rank six quotient vector bundle on $G(2,V_8)$. Thus, in this case, we need to look at two-forms in six variables.

If $V_6$ is as before a six-dimensional complex vector space, 
$\wedge^4 V_6^\vee\simeq \wedge^2 V_6$ has only two proper $\GL(V_6)$-orbits
closures, that we will index by their codimension: the Pfaffian cubic 
hypersurface $Z_1$ and its singular locus $Z_6$, that is the cone over the Grassmannian $G(2,V_6)$. 
These allow us to construct inside $G(2,V_8)$ the two orbital degeneracy loci $D_{Z_1}(v)$ and $D_{Z_6}(v)$.

\smallskip

Let us first consider $D_{Z_6}(v)$, which can also be defined by
\[
D_{Z_6}(v):=\Big\{[U_2]\in G(2,V_8)\mid \exists U_6 \supset U_2, \;\;
v \in \wedge^3 V_8\wedge U_2+\wedge^4 U_6\Big\}.
\]

\begin{lemma}
$D_{Z_6}(v)$ is a smooth Fano sixfold of even index.
\end{lemma}

\proof By definition, $D_{Z_6}(v)$ is the projection in $G(2,V_8)$ of the locus 
$Z_6(v)$ in $Fl(2,6,V_8)$ parametrizing flags $(U_2\subset U_6\subset V_8)$
such that $v$ belongs to the $56$-dimensional space 
$\wedge^3 V_8\wedge U_2+\wedge^4 U_6$. Taking the quotient of $\wedge^4V_8$ by the latter, we get a rank $14$ vector bundle $\cP$ on $Fl(2,6,V_8)$. Moreover $v$ defines a 
generic section of this vector bundle and $Z_6(v)$ is the zero-locus of 
this section. Since $Fl(2,6,V_8)$ has dimension $20$, we deduce that $Z_6(v)$
is smooth of dimension $6$, and that 
its canonical bundle is given by the adjunction formula. 
A straightforward computation yields
$$K_{Z_6(v)}=\det(U_2)^{-3}\otimes \det(U_6)^{5}.$$
On the other hand, for any $[U_2]\in G(2,V_8)$, the quotient of $\wedge^4V_8$ by 
 $\wedge^3 V_8\wedge U_2$ is isomorphic to $\wedge^4(V_8/U_2)\simeq 
\wedge^2(V_8/U_2)^\vee\otimes \det (V_8/U_2)$. This is a space of skew-symmetric forms in
six dimensions, and the existence of $U_6$ exactly means that $v$ defines a 
skew-symmetric form in  $\wedge^2(V_8/U_2)^\vee\otimes \det (V_8/U_2)$ whose rank
is at most two. In fact the rank must be exactly two, since for $v$ generic, a simple dimension count shows that the rank can never be zero. In particular the projection 
of $Z_6(v)$ to $D_{Z_6}(v)$ is an isomorphism. 

More than that, the kernel of our two form on $V_8/U_2$ is $U_6/U_2$, so we get a
non-degenerate skew-symmetric form on the quotient $V_8/U_6$, which is therefore 
identified with its dual. To be precise, since the skew-symmetric form has values in 
 $\det (V_8/U_2)$, we get an isomorphism  $V_8/U_6\simeq (V_8/U_6)^\vee\otimes 
 \det (V_8/U_2)$. Taking determinants, we deduce that $\det (U_2)^2\simeq \det (U_6)^2$;
 in other words, the line bundle $\cL=\det (U_6)\otimes \det (U_2)^\vee$ is $2$-torsion 
 on  $Z_6(v)$.
 
But then we can rewrite the canonical bundle as 
$$K_{Z_6(v)}=\det(U_2)\otimes \det(U_6)\otimes \cL^{\otimes 4}.$$
Note that $\det(U_2)^\vee\otimes \det(U_6)^\vee$ is very ample on $Fl(2,6,V_8)$ since it defines its canonical Pl\"ucker type embedding. Since $\cL$ is torsion we deduce that 
 $Z_6(v)$ is Fano. But then its Picard group is torsion free, so $\cL$ is actually trivial. So finally $K_{Z_6(v)}=\det(U_2)^2$, hence the index is even. \qed

\medskip The previous discussion shows that $D_{Z_6}(v)$ is a Pfaffian locus 
defined by a skew-symmetric map $\psi_v : \cQ\ra \cQ^\vee (1)$ associated with $v$. The rank four sheaf $\mathcal{K}er(\psi_v)$
(which is $\cU_6/\cU_2$ in the previous proof) fits into an exact sequence 
\begin{equation}\label{esker}
0\ra \mathcal{K}er(\psi_v)\lra \cQ\lra \cQ^\vee (1)\lra \mathcal{K}er(\psi_v)^\vee (1)\lra 0.
\end{equation}

Let us set once and for all the more compact notation $D:=D_{Z_6}(v)$ and $G:=G(2,V_8)$. The exact sequence \eqref{esker} allows to describe the normal bundle $\cN_{D/G}$ as follows. 

\begin{lemma}\label{normal} We have isomorphisms
$\cN_{D/G}\simeq \wedge^2\mathcal{K}er(\psi_v)^\vee (1)$, and $\cN_{D/G}^\vee\simeq \cN_{D/G}(-2)$.
\end{lemma}

We will use this information later on. 
Our next goal is to prove the following theorem.

\begin{theorem}
\label{thm_main}
For a generic $v\in \wedge^4 V_8$ the orbital degeneracy locus $D_{Z_6}(v)$ is isomorphic with the moduli space $\SU_C(2,\cO_C(c))$ of semistable rank two vector bundles on $C$ 
with fixed determinant $\cO_C(c)$, for a certain point $c\in C$.
\end{theorem}

\begin{remark}
Such embeddings defined by Hecke lines are studied in \cite[section 3.4]{beauville}, and there is 
one, denoted $\varphi_p$ in loc. cit., for each choice of a point $p$ on the curve $C$. 
Here we only 
get one of these embeddings, in agreement with the already mentionned fact that $v$ does not only 
determine a genus three curve, but a marked point on this curve.
\end{remark}

\begin{remark}
An interesting consequence is that we know a minimal resolution of the structure sheaf of $\SU_C(2,\cO_C(c))$ inside the Grassmannian $G(2,V_8)$. From this resolution it is easy to check that 
the intersection with a general copy of $G(2,6)$ inside $G(2,V_8)$ is a K3 surface of genus $13$. 
This kind of description is used in \cite{km} to provide a new model for the general such K3 surface.
\end{remark}

\smallskip
Let us begin by showing that $D_{Z_6}(v)$ defines a six-dimensional family of Hecke lines.

\begin{prop}
\label{prop_hecke_in_coble}
Let $[U_2]\in D_{Z_6}(v)$, then $\PP(U_2)\subset \PP(V_8)$ is a line  in $\cC_4$.
\end{prop}

\begin{proof}
Let $[U_1]\in \PP(U_2)$ be a point in the line. By definition of $D_{Z_6}(v)$, one can write $$ (v  \mod U_1)=u_2\wedge v'+ a\wedge b\wedge c\wedge d$$ for some $u_2\in U_2$, some trivector $v'$ and some vectors $a,b,c,d$. The trivector $v'$ is a trivector in six variables, therefore it can in general be written as $e\wedge f\wedge g + h\wedge i\wedge l$ for some vectors $e,f,g,h,i,l$, since the secant variety of $G(3,6)$ fills the full Plücker space. Now, modulo $U_2$, $\dim(\langle a,b,c,d \rangle \cap \langle e,f,g \rangle)\geq 1$ and $\dim(\langle a,b,c,d \rangle \cap \langle h,i,l \rangle)\geq 1$. Thus we can suppose that $a=e$ and $b=h$. But then if we let  $U_4=\langle U_2,a,b \rangle,$  it is straightforward to check that $( v \mod U_1) \in (\wedge^2 U_4)\bigwedge (\wedge^2 V_8)$. This ensures that $[U_1]$ belongs to $\cC_4$. 
\end{proof}

The point-line incidence variety  of the family of lines parametrized by $D_{Z_6}(v)$ is given by the projective bundle $\PP(\cU_2)\to D_{Z_6}(v)$.

\begin{prop}
The family of lines parametrized by $D_{Z_6}(v)$ covers $\cC_4$.
\end{prop}

\begin{proof}
This is again a Chern class computation. Indeed, by irreducibility of the varieties in play, it is sufficient to check that, if $\cU_1^\vee$ denotes the relative dual tautological line bundle of $\PP(\cU_2)\to D_{Z_6}(v)$, then $c_1(\cU_1^\vee)^6\neq 0$. This implies that the image of $\PP(\cU_2)$ inside $\PP(V_8)$ has dimension at least six, and is thus the Coble quartic $\cC_4$ by Proposition \ref{prop_hecke_in_coble}. Notice that one can work directly on $Z_6(v)$, since it is isomorphic to $D_{Z_6}(v)$. Since $Z_6(v)$ can be constructed as the zero locus of a section of a vector bundle inside the flag variety $Fl(2,4,V_8)$, we can verify that $c_1(\cU_1^\vee)^6\neq 0$  with \cite{Macaulay2} by constructing the coordinate ring of the zero locus $Z_6(v)$ and of the projective bundle $\PP(\cU_2)$ over it, similarly to what we did in the proof of Lemma \ref{U4covers}.
\end{proof}

\begin{prop}
The lines parametrized by $D_{Z_6}(v)$ are Hecke lines.
\end{prop}

\proof 
Suppose by contradiction 
that the lines parametrized by  $D_{Z_6}(v)$ are not Hecke. Since they form a covering family, they must be lines in the ruling, i.e. $D_{Z_6}(v)\subset \cF_R$. Now recall 
that $\cF_R$ is a birational image of the quadric bundle $G(2,\cU_4)$ over $\cA$.  The pre-image of $D_{Z_6}(v)$ in $G(2,\cU_4)$ is rationally connected, being birationally equivalent to the Fano manifold $D_{Z_6}(v)$. But then its projection to $\cA$ must be constant. Since the fibers of this projection are only four-dimensional, while the
dimension of  $D_{Z_6}(v)$ is six, we get a contradiction.
\qed 

\medskip\noindent {\it Proof of Theorem \ref{thm_main}.}
Recall that the family $\cF_H$ of Hecke lines has dimension seven, 
so  $D_{Z_6}(v)$ cannot 
be the whole family. In fact $\cF_H$ has a rational map $\eta$ to $C$, and by the same argument as above, the fact that $D_{Z_6}(v)$ is Fano ensures that its image in $\cF_H$ 
is contained in a fiber of $\eta$, over some point $c\in C$. But then the morphism
from $\SU_C(2,\cO_C(c))$ to $\cF_H$ is birational onto its image $D_{Z_6}(v)$.
Since $\SU_C(2,\cO_C(c))$ has Picard rank one \cite{ramanan}, this morphism must be an
isomorphism. \qed

\section{A Coble type quadric  hypersurface}

The aim of this section is to show that the Coble quadric hypersurface in $G(2,V_8)$ deserves its name, in the sense that it is singular along the moduli space and it is uniquely determined by this property. So the section is mainly devoted to the proof of Theorem \ref{Coble_G(2,8)}. In the last part we also prove a self-duality statement concerning this hypersurface which is analogous to the self-duality of the Coble quartic 
in $\PP(V_8)$.

\subsection{The relative Pfaffian}
As we have seen, the fact that $D_{Z_6}(v)$ is defined as a Pfaffian locus 
in $G(2,V_8)$ implies that it is the singular locus of a Pfaffian hypersurface,
defined as the first degeneracy locus $D_{Z_1}(v)$ of the skew-symmetric morphism 
$\cQ\lra \cQ^\vee(1)$ defined by $v$. 

\begin{theorem}\label{Coble_G(2,8)}
The hypersurface $D_{Z_1}(v)$ of $G(2,V_8)$ is a quadratic section of the Grassmannian.
It is the unique quadratic section that is singular along  $D_{Z_6}(v)$.
\end{theorem}

\begin{remark}
Starting from a genus three curve $C$ and its Kummer threefold embedded in $\PP^7$ by the 
linear system $|2\Theta|$, the original observation 
of Coble was that there exists a unique Heisenberg-invariant quartic $\cC$ that is singular 
along the Kummer. Beauville proved much later that the Heisenberg-invariance hypothesis was 
actually not necessary \cite{beauville-coble}. In our context the curve and its 
Heisenberg group are not easily available (although there are connections between the 
latter and the  Weyl group $W(E_7)$ of the theta-representation $\wedge^4V_8$), so we do not use any Heisenberg-invariance hypothesis. 
\end{remark}

\subsubsection{Structure of the proof of Theorem \ref{Coble_G(2,8)}} Let us write $D$ for $D_{Z_6}(v)$ and $G$ for $G(2,V_8)$, for simplicity. 
That $D_{Z_1}(v)$ is a quadratic section of $G$ follows from the fact that 
it is defined by a rank six Pfaffian, obtained as the image of $v$ by the 
cubic morphism 
$$S^3(\wedge^2\cQ^\vee(1))\ra \wedge^6\cQ^\vee(3)= \cO_G(2).$$ 

In order to prove that this is the only quadratic section that is singular along $D$, 
recall that the conormal bundle
of $D$ in the Grassmannian $G$ is the quotient of the ideal sheaf $\cI_D$
by its square  $\cI^2_D$. Twisting by $\cO_G(2)$ and taking cohomology, we get an 
exact sequence 
$$0\lra H^0(G,\cI^2_D(2))\lra H^0(G,\cI_D(2))\lra H^0(D,\cN_{D/G}^\vee(2))\lra H^1(G,\cI^2_D(2)).$$
Observe that $H^0(G,\cI_D(2))$ parametrizes quadratic sections of $G$ (up to scalar)
that contain $D$, while, since $D$ is smooth, 
$H^0(G,\cI^2_D(2))$ parametrizes quadratic sections that are
singular along $D$. Our claim is that the latter space is one-dimensional. This will 
be proved in three steps: first, compute the dimension of the space of quadrics 
containing $D$; second, bound $H^0(D,\cN_{D/G}^\vee(2))$ from below; third, prove that 
$H^1(G,\cI^2_D(2))$ vanishes. These results are contained in Lemmas \ref{lem_I}, \ref{lem_N}, \ref{lem_II}. From the fact that $H^1(\cI_D^2(2))=0$, 
the exact sequence 
$$0\lra H^0(\cI_D^2(2))\lra H^0(\cI_D(2)) \lra H^0(\cN_{D/G}^\vee(2))= H^0(\cN_{D/G})\lra 0,$$
knowing that $h^0(\cI_D(2))=71$ and  $h^0(\cN_{D/G})\ge 70$, will allow us to 
conclude that $h^0(\cI_D^2(2))\le 1$ and the proof will be complete. 
\subsubsection{Quadrics containing the moduli space}
Let us count the quadric sections of $G=G(2,V_8)$ that contain the moduli space 
$D\simeq \SU_C(2,L)$.

\begin{lemma}
\label{lem_I}
$h^0(G,\cI_D(2))=71.$
\end{lemma}

\proof Let us first recall the classical minimal 
resolution of the ideal $I$ generated by submaximal Pfaffians of a generic skew-symmetric matrix of size $6$; in other words, the ideal of the 
cone over the Grassmannian $G(2,V_6)$ inside $\wedge^2V_6$. Letting
$S=\CC[\wedge^2V_6]$, this resolution is the following \cite{jpw}:

$$\xymatrix@-2.5ex{
 & 0 & \\ & I\ar[u] & \\ &  \wedge^4V_6^\vee\otimes S(-2){\bf \ar[u]} & \\
 & S_{21111}V_6^\vee\otimes S(-3)\ar[u] & \\ 
 S_{311111}V_6^\vee\otimes S(-4)\ar[ru] &\oplus &S_{22222}V_6^\vee\otimes S(-5)\ar[lu]\\
 & S_{322221}V_6^\vee\otimes S(-6)\ar[ru]{\bf \ar[lu]}& \\
 & S_{332222}V_6^\vee\otimes S(-7)\ar[u] & \\
 & \det(V_6^\vee)^3\otimes S(-9){\bf \ar[u]} & \\
& 0\ar[u] &
}$$

Since $D$ is a Pfaffian locus of the expected dimension, given by a skew-symmetric map 
$\cQ\ra\cQ^\vee(1)$, we deduce the following 
free resolution of its twisted ideal sheaf (we used identifications like $S_{332222}\cQ^\vee=\wedge^2\cQ^\vee(-2)$) :

$$\xymatrix@-2.5ex{
 & 0 & \\ & \cI_D(2)\ar[u] & \\ &  \wedge^4\cQ\ar[u] & \\
 &  \fsl (\cQ)\ar[u] & \\ 
 S^2\cQ^\vee(-1) \ar[ru] &\oplus & S^2\cQ(-1) \ar[lu]\\
 & \fsl (\cQ)(-2)\ar[ru]\ar[lu] & \\
 & \wedge^2\cQ(-3)\ar[u] & \\
 & \cO_G(-4)\ar[u] & \\
& 0\ar[u] &
}$$
Note that this resolution is self-dual, up to twist. Moreover, using the Bott-Borel-Weil theorem 
one can check that all the factors are acyclic homogeneous vector bundles, with two exceptions:
$\wedge^4\cQ$ has a non zero space of sections, isomorphic to $\wedge^4V_8$; and  $S^2\cQ^\vee(-1)$,
which is one of the two irreducible factors of $\Omega_G^2$, has a one dimensional cohomology group in degree two. We end up with a canonical exact sequence
\begin{equation}\label{71}
0\lra \wedge^4V_8\lra H^0(G,\cI_D(2))\lra \CC\lra 0,
\end{equation}
and our claim follows.
\qed 

\begin{remark}
Being defined by a cubic Pfaffian, the equation of the 
hypersurface $D_{Z_1}(v)$ must be a cubic $\SL(V_8)$-covariant of $v$ in $\wedge^4V_8$, 
taking values in $H^0(\cO_G(2))\simeq S_{22}V_8$. In fact it is a $\GL(V_8)$-covariant, 
that by homogeneity with respect to $V_8$, must take its values in $S_{22}V_8\otimes\det(V_8)$. One can check that 
the latter module has multiplicity one inside $S^3(\wedge^4V_8)$, so this covariant is unique 
up to scalar. For example, it can be obtained as the composition 
\begin{multline*}
S^3(\wedge^4V_8)\hookrightarrow  S^3(\wedge^2V_8\otimes \wedge^2V_8)\to 
S^3(\wedge^2V_8)\otimes S^3(\wedge^2V_8)\to S^3(\wedge^2V_8)\otimes \wedge^6V_8 \to \\
\to S^3(\wedge^2V_8)\otimes \wedge^2V_8^\vee\otimes\det(V_8)
\to S^2(\wedge^2V_8)\otimes\det(V_8)
\to S_{22}V_8\otimes\det(V_8).    
\end{multline*}

Following the natural morphisms involved in these arrows, this would allow to give an
explicit formula for an equation of the quadratic hypersurface $D_{Z_1}(v)$ in terms 
of the coefficients of $v$ (this was done in \cite{rsss} for the Coble quartic itself). 
It would suffice to do this when $v$ belongs to our prefered 
Cartan subspace; this is in principle a straightforward computation but the resulting 
formulas would be huge.
\end{remark}

\begin{remark}
The embedding of $\wedge^4V_8$ inside $H^0(G,\cI_D(2))$ in (\ref{71}) is given by the 
derivatives of $D_{Z_1}(v)$ with respect to $v$, that is, can be obtained 
by polarizing the cubic morphism discussed in the previous remark. On the other hand, modulo these derivatives, (2)
shows that there is a uniquely defined "non-Pfaffian"  quadric vanishing on $D$. 
This non-Pfaffian quadric comes from the contribution of $S^2\cQ^\vee(-1)$ in 
the resolution of $\cI_D(2)$. Since in this resolution, these two terms are connected 
one to the other through three morphisms having respective degree two, 
one, and two with respect to $v$,  the non-Pfaffian quadric 
must be given by a {\it quintic} covariant in $v$. And indeed, a computation with LiE  \cite{LiE} shows that 
$$\Hom(S^5(\wedge^4 V_8), S_{22}V_8\otimes (\det (V_8))^2)^{\GL(V_8)}\simeq\CC^2.$$
A special line in this space of covariants is generated by the cubic covariant defining 
the Pfaffian quadric, twisted by the invariant quadratic form (defined by the wedge product). The quotient is our 
non-Pfaffian quadric. As before we could in principle compute it explicitely by 
constructing a specific covariant. One way to construct such a covariant is to observe
that 
$$S^2(\wedge^4 V_8)\supset S_{221111}V_8\subset \wedge^2 V_8\otimes \wedge^6 V_8=
\wedge^2 V_8\otimes \wedge^2 V_8^\vee\otimes\det (V_8).$$
Taking the square of the resulting morphism we can define a quartic covariant
$$S^4(\wedge^4 V_8)\to S^2(\wedge^2 V_8)\otimes S^2(\wedge^2 V_8^\vee)\otimes\det (V_8)^2\to S_{22}V_8\otimes \wedge^4 V_8^\vee\otimes\det (V_8)^2,$$
hence the desired  quintic covariant. 
\end{remark}

\subsubsection{The normal bundle of $D$ in $G(2,V_8)$}

Let us now bound from below the dimension of 
$H^0(D,\cN_{D/G}^\vee(2))$. By Lemma \ref{normal}, this space is isomorphic with $H^0(\cN_{D/G})$, which 
parametrizes infinitesimal deformations of $D$ inside $G$. Some of these deformations 
must be induced by the deformation of $[v]$ inside $\PP(\wedge^4V_8)$, which 
should provide $69$ parameters. But recall that the family $\cF_H$ of Hecke lines inside $SU_C(2)$ is a subvariety of 
$G(2,V_8)$, birationally fibered over the curve $C$, with one fiber isomorphic to 
$D\simeq \SU_C(2,\cO_C(c))$ for some point $c\in C$.  So we expect one extra deformation 
of $D$ to be obtained by deforming $c$ in the curve $C$. That these deformations are independent 
is essentially the content of

\begin{lemma}
\label{lem_N}
$h^0(D,\cN_{D/G})\ge 70.$
\end{lemma}

\proof
The locus in $\wedge^4 V_6\simeq\wedge^2V_6^\vee$ corresponding to skew-symmetric forms of rank at most $2$ is desingularized by the total space of $\wedge^4 \cU_4 $ over the Grassmannian $G(4,V_6)$. As a consequence of this and of \cite[Proposition 2.3]{bfmt}, the Pfaffian locus $D$ is desingularized by the zero locus $Z:=Z_6(v)$ inside $Fl(2,6,V_8)$ of a (general) section of the  bundle $\cV=\wedge^4 (V_8/\cU_2)/\wedge^4 (\cU_6/\cU_2)$. This bundle is an extension of irreducible bundles 
$$ 0 \to \wedge^3 (\cU_6/\cU_2)\otimes (V_8/\cU_6) \to \cV \to \wedge^2 (\cU_6/\cU_2)\otimes \det(V_8/\cU_6) \to 0.$$
By dimension count,
$Z$ is in fact isomorphic to $D$ via the natural projection. Under this isomorphism and 
by Lemma \ref{normal}, $\cN_{D/G}$ can be identified with the restriction of $\cN:=\wedge^2 (\cU_6/\cU_2)\otimes \det(V_8/\cU_6)$ to $Z$. In order to compute the cohomology of this restriction we can tensorize with $\cN$ the Koszul complex $\wedge^\bullet \cV^\vee$ of the global section of $\cV$, whose zero locus is 
 $Z\subset Fl(2,6,V_8)$. This gives the following resolution of $\cN_{D/G}$ by locally free sheaves on $Fl(2,6,V_8)$
$$ 0\to \wedge^\bullet \cV^\vee \otimes \cN \to \cN_{D/G} \to 0.$$
By applying the Bott-Borel-Weil Theorem we can compute the cohomology groups of the bundles $\wedge^k \cV^\vee \otimes \cN$, for all $k \ge 0$. Those 
that do not vanish are the following:
\begin{gather*}
H^0(\wedge^0 \cV^\vee \otimes \cN)=\wedge^4 V_8, \\
H^0(\wedge^1 \cV^\vee \otimes \cN)=\CC,\\
H^2(\wedge^3 \cV^\vee \otimes \cN)=\CC , \;\; H^3(\wedge^3 \cV^\vee \otimes \cN)=\CC^2, \\
H^4(\wedge^4 \cV^\vee \otimes \cN)=  H^5(\wedge^4 \cV^\vee \otimes \cN)=\wedge^4 V_8,\\
H^4(\wedge^5 \cV^\vee \otimes \cN)=\fsl(V_8)\oplus \CC^3, \quad  H^5(\wedge^5 \cV^\vee \otimes \cN)=\fsl(V_8)\oplus \CC^4 , \quad H^6(\wedge^5 \cV^\vee \otimes \cN)=\CC,\\
H^6(\wedge^7 \cV^\vee \otimes \cN)= H^7(\wedge^7 \cV^\vee \otimes \cN)=\CC,\\
H^8(\wedge^9 \cV^\vee \otimes \cN)= H^9(\wedge^9 \cV^\vee \otimes \cN)=\CC,\\
H^{12}(\wedge^{13} \cV^\vee \otimes \cN)= H^{13}(\wedge^{13} \cV^\vee \otimes \cN)=
\CC.
\end{gather*}

\medskip
A direct consequence is that $\chi(\cN_{D/G})=70$. Moreover, observe that 
$$H^q(\wedge^k \cV^\vee \otimes \cN)=0\quad \mathrm{for} \;  q-k>1.$$ 
Since these groups give the first page of the spectral sequence in cohomology 
induced by the Koszul complex of $\cO_{Z}$ twisted by $\cN$,  this implies that 
$H^i(\cN_{D/G})=0$ for $i>1$.  
Therefore $h^0(\cN_{D/G})=\chi(\cN_{D/G})+h^1(\cN_{D/G})\ge 70$.\qed

 \subsubsection{An affine module $M$}
As usual $V_6$ denotes a six dimensional vector space. The ideal $I$ of the cone over $G(2,V_6)$ is generated
by the submaximal Pfaffians of the generic skew-symmetric matrix of size $6$; the $\GL(6)$-module generated by these submaximal Pfaffians is $\wedge^4V_6^\vee\subset 
S^2(\wedge^2V_6^\vee)$. The square of $I$ is then generated by the symmetric square of this module, which decomposes as 
$$S^2(\wedge^4V_6^\vee) = S_{221111}V_6^\vee\oplus S_{2222}V_6^\vee.$$
The first component is $\wedge^2V_6^\vee\otimes \det V_6^\vee$, and must be  
interpreted as parametrizing quartics that are multiples of linear forms by the Pfaffian 
cubic. The ideal they generate is $S_+I_{P}$, where $S_+\subset S$ is the irrelevant 
ideal, and $I_P$ denotes the ideal of the Pfaffian hypersurface. 

Consider the exact sequence 
$$ 0\to S_+I_P \to I^2 \to M:=I^2/ S_+I_P \to 0.$$ 
The quotient module $M$ is generated by $S_{2222}V_6^\vee$.
According to \cite{Macaulay2}, the minimal resolution 
$R_\bullet$ of $M$ has the Betti numbers of Table \ref{tab:table1}.

\begin{table}[h]
    \[\begin{array}{l|cccccc}
         & 0 & 1 & 2 & 3 & 4 & 5 \\ \hline
       0: & \cdot & \cdot & \cdot & \cdot & \cdot & \cdot  \\
      1: & \cdot & \cdot & \cdot & \cdot & \cdot & \cdot  \\
      2: & \cdot & \cdot & \cdot & \cdot & \cdot & \cdot  \\
      3: & \cdot & \cdot & \cdot & \cdot & \cdot & \cdot  \\
      4: & 105 & 399 & 595 & 405 & 105 & \cdot  \\
      5: & \cdot & \cdot & \cdot & 21 & 35 & 15  
    \end{array}
\]

\caption{Betti table of $M$}
\label{tab:table1}
\end{table}

The minimal resolution is $\GL_6$-equivariant and it is not difficult to 
write it in terms of Schur functors. Indeed, we know that the quartic generators
are parametrized by $S_{2222}V_6^\vee$, so the first syszygy module must be 
contained in $S_{2222}V_6^\vee\otimes \wedge^2V_6^\vee$, and it turns 
out that there is a unique $\GL_6$-module of the correct dimension inside this tensor 
product. Proceeding inductively we arrive at the following conclusion: the minimal
$\GL_6$-equivariant resolution of the $\bS$-module $M$ has the following shape:
$$\xymatrix@-2ex{
 0 & \\
 M\ar[u] & \\
 S_{2222}V_6^\vee \otimes S(-4)\ar[u] & \\
 (S_{32221}V_6^\vee\oplus\ar[u] {\bf S_{222211}}V_6^\vee) \otimes S(-5) & \\
 (S_{422211}V_6^\vee\oplus S_{33222}V_6^\vee\ar[u] \oplus {\bf S_{322221}}V_6^\vee)\otimes S(-6) & \\
 (S_{432221}V_6^\vee\oplus\ar[u] {\bf S_{422222}}V_6^\vee)\otimes S(-7) & {\bf S_{333331}}V_6^\vee \otimes S(-8)\ar[lu] \\
 S_{442222}V_6^\vee\otimes S(-8)\ar[u]  & {\bf S_{433332}}V_6^\vee\otimes S(-9)\ar[lu]\ar[u] \\
  & {\bf S_{443333}}V_6^\vee\otimes S(-10)\ar[lu]\ar[u] \\
 & 0\ar[u] }$$
 Here vertical arrows have degree one and 
 diagonal arrows have degree two. Notice that the complex in bold reproduces the 
 resolution of the Pfaffian ideal $I$ itself. 
 
 \begin{remark}
 As J. Weyman observed, one could also obtain this resolution by considering the 
 natural resolution of the Pfaffian hypersurface given by the total space of 
 the vector bundle $\wedge^2\cU$ over the Grassmannian $G(4,V_6)$. The morphism
 $\pi$ from $\Tot(\wedge^2\cU)$ to $\wedge^2V_6$ is a resolution of singularities, 
 and one can check that $M$ is the push-forward by $\pi$ of the module given by 
 the pull-back of the line bundle $\cO(2)$ from the Grassmannian. Applying the 
 geometric technique from \cite{weyman-book}, one can extract the minimal resolution 
 of $M$ from the collection of $\GL(V_6)$-modules given by
 $$F_i = \bigoplus_{j\ge 0} H^j(G(4,V_6),\cO(2)\otimes\wedge^{i+j}(\wedge^2\cU)^\perp).$$
 Here $(\wedge^2\cU)^\perp$ is the kernel of the natural projection $\wedge^2V_6^\vee \to \wedge^2\cU ^\vee$. The bundle $(\wedge^2\cU)^\perp$ is not semisimple 
 but is an extension of $\cO(-1)$ by $\cU^\vee\otimes Q^\vee$. Remarkably, it is the 
 contribution of $\cO(-1)$ that reproduces the minimal resolution of $I$ (twisted) 
 inside that of $M$. 
 \end{remark}
 
\subsubsection{Relativizing $M$} Now we want to use these results in the 
relative setting. Since $\wedge^4 \cQ$ is a vector bundle on $G(2,V_8)$ which is locally isomorphic to $\wedge^2 V_6$, we can relativize the construction of $I$ and $I_P$ 
and $M$. For convenience let us restrict to the complement $\cX$ of the zero section inside the total space of this vector bundle.  
We get sheaves of $\cO_{\cX}$-modules and ideals that we denote respectively 
by $\cI', \cI'_P, \cM'$. Note that since we avoid the zero section, 
we get an exact sequence 
$$ 0\to \cI'_P \to \cI'^2 \to \cM' \to 0.$$

Then we consider $v\in\wedge^4V_8$ as a general section of $\wedge^4 \cQ$, 
that we 
interpret as a morphism from $G=G(2,V_8)$ to the total space of $\wedge^4 \cQ$. By the definition of orbital degeneracy loci \cite[Definition 2.1]{bfmt2}, the ideal of $D_{Z_1}(v)$ is  $\cI_P:=\cI'_P\otimes\cO_G$ and the ideal of $D=D_{Z_6}(v)$ is  $\cI_D:=\cI'\otimes\cO_G$.
Let us also denote $\cM=\cM'\otimes\cO_G$. Of course these tensor produts are 
taken over $\cO_{\cX}$.

\begin{lemma}\label{exact}
There is an exact sequence
$$ 0\to \cI_P \to \cI_D^2 \to \cM \to 0.$$ 
\end{lemma}

\begin{proof}
By the right exactness of tensor product, here by $\cO_G$, 
we get an exact sequence 
$$ \cI_P \to \cI_D^2 \to \cM \to 0.$$ 
But the map $\cI_P \subset \cI_D^2$ (which expresses the fact that $D$ is 
contained in the singular locus of the Pfaffian hypersurface) clearly remains 
an injection, and we are done.
\end{proof}

In order to control $\cM$ we will now  consider the complex of vector 
bundles induced by the resolution we constructed for $M$. We can deduce a resolution of  $\cM'$ and then tensor out again by $\cO_G$. 
In order to prove that we get  a resolution  of $\cM$ (the resolution given just below), 
we need to check that the 
Tor-sheaves of $\cO_\cX$-modules $\mathcal{T}or_i(\cM',\cO_G)$ vanish for $i>0$.
All the Tor-sheaves we compute in the sequel will also be for $\cO_\cX$-modules.
$$\xymatrix@-2ex{
 0 & \\
 \cM\ar[u] & \\
 S_{2222}\cQ(-4)\ar[u] & \\
 S_{32221}\cQ(-5)\oplus\ar[u] S_{222211}\cQ(-5) & \\
 S_{422211}\cQ(-6)\oplus S_{33222}\cQ(-6)\ar[u] \oplus S_{322221}\cQ(-6) & \\
 S_{432221}\cQ(-7)\oplus \ar[u] S_{422222}\cQ(-7) & S_{333331}\cQ(-8)\ar[lu] \\
 S_{442222}\cQ(-8)\ar[u]  & S_{433332}\cQ(-9)\ar[lu]\ar[u] \\
  & S_{443333}\cQ(-10)\ar[lu]\ar[u] \\
 & 0\ar[u] }$$

\begin{lemma} For any $i>0$, 
\begin{enumerate}
    \item $\mathcal{T}or_i(\cI'_P,\cO_G)=0,$
    \item $\mathcal{T}or_i(\cO_\cX/\cI',\cO_G)=0,$
    \item $\mathcal{T}or_i(\cI'/\cI'^2,\cO_G)=0,$
    \item $\mathcal{T}or_i(\cM',\cO_G)=0.$
\end{enumerate}
\end{lemma}

\begin{proof}
(1) is obvious since $\cI'_P$ is locally free. (2) is a consequence of the 
Generic Perfection Theorem (see \cite{EN}), since $I$ and therefore $\cI'$ is perfect, and $D$ has 
the expected dimension. (3) is a consequence of (2), because $\cI'/\cI'^2$ is 
a locally free $\cO_\cX/\cI'$-module (recall that since we have a generality assumption the singular locus is avoided). 
Finally to prove (4) observe first that by the long exact sequence of Tor's,  $\mathcal{T}or_i(\cI',\cO_G)=\mathcal{T}or_ {i+1}(\cO_\cX/\cI',\cO_G)=0$ for any $i>0$. 
Because of (3) this implies that $\mathcal{T}or_i(\cI'^2,\cO_G)=0$ for any $i>0$. 
Then we can use the exact sequence of Lemma \ref{exact} to deduce that 
$\mathcal{T}or_i(\cM',\cO_G)=0$ when $i>1$, and that there is an exact sequence 
$$0\lra \mathcal{T}or_1(\cM',\cO_G)\lra \cI_P\lra \cI_D^2\lra \cM\lra 0.$$
By Lemma \ref{exact}, $\mathcal{T}or_1(\cM',\cO_G)$ vanishes as well, and we are done.\end{proof}

\begin{lemma}\label{acyclic}
$\cM(2)$ is acyclic.
\end{lemma}

\begin{proof} Twist the previous resolution of $\cM$ by $\cO(2)$ and deduce from the 
Bott-Borel-Weil theorem that all the bundles in the twisted resolution are acyclic. 
This implies the claim.
\end{proof}

\begin{lemma}
\label{lem_II}
$H^i(\cI_D^2(2))=0$ for any $i>0$. 
\end{lemma}

\begin{proof} This follows immediately from Lemmas \ref{exact} and \ref{acyclic}.
\end{proof}

This concludes the proof of Theorem \ref{Coble_G(2,8)}. Note the following consequence: $D$ has non-obstructed deformations.

\begin{coro}\label{cor_normal_coho} $h^0(\cN_{D/G})=70$ and $h^i(\cN_{D/G})=0$ for any $i>0$.
\end{coro}

\subsection{Deforming the Pfaffian hypersurface}
We already observed that varying $v$ in $\wedge^4V_8$, 
we only get a codimension one family of deformations of $D$. The missing 
dimension is provided by the choice of the point on the 
curve $C$, but this is invisible in our constructions. We will nevertheless prove 
that the special quadric section of the Grassmannian deforms. 

\begin{lemma}
For a generic point $p\in C$, and the associated embedding $\varphi_p: \SU_C(2,\mathcal{O}_C(p)) \hookrightarrow G(2,V_8)$, there exists at most one quadric hypersurface $Q_p$ in 
the Grassmannian, that is singular along $\SU_C(2,\mathcal{O}_C(p))$.
\end{lemma}

\begin{proof}
Such a quadric corresponds to a line in $H^0(G(2,V_8), \cI^2_{\SU_C(2,\mathcal{O}_C(p))}(2))$ and we have computed in the proof of Theorem \ref{Coble_G(2,8)} that this space
has dimension one for certain special points $p$. By semicontinuity this dimension remains smaller or equal to one for $p$ generic. 
\end{proof}

\begin{theorem}\label{quademb}
For the generic embedding $\varphi_p$, there exists a unique quadric hypersurface 
of $G(2,V_8)$ that is singular along $\SU_C(2,\mathcal{O}_C(p))$.
\end{theorem}

\begin{proof}
Let us consider the embedding $Q=D_{Z_1}(v)\hookrightarrow G$ from Theorem \ref{Coble_G(2,8)}.
Let $H'_{Q/G}$ be the so-called "locally trivial Hilbert scheme" parametrizing locally trivial deformations of $Q\subset G$, as defined in \cite[2.2]{GK89}. Remark that the construction of \cite{GK89} is done for finite singularities, but their arguments, as the authors underline in the introduction, 
go through for arbitrary singularities because of \cite{FK87}. Let 
$$\mathcal{N}'_{Q/G}= \ker ( \mathcal{N}_{Q/G} \to \mathcal{T}^1_Q),$$ where $\mathcal{T}^1_Q$ denotes the first cotangent sheaf of $Q$ (as defined, for instance, in \cite[Section 1.1.3]{Sernesi}). 
In order that 
the  locally trivial Hilbert scheme be smooth at $Q$, by \cite[Prop. 2.3]{GK89} we need that $H^1(Q,\mathcal{N}'_{Q/G})=0$. If this happens, then $h^0(Q,\mathcal{N}'_{Q/G})=\dim(H'_{Q/G})$ and we will show that this equals 70.
By \cite[Section 4.7.1]{Sernesi} we have an exact sequence 
$$ 0 \to T_Q \to T_{G}|_Q \to \cN_{Q/G} \to \mathcal{T}^1_Q \to 0 .$$
Hence  $\mathcal{N}'_{Q/G}$ coincides with the image of $T_{G/Q}$ inside $\cN_{Q/G}$, which  is exactly the (twisted) jacobian ideal $\cJ_{Q/G}(2)$ restricted to $Q$. In turn, the Jacobian ideal of the Pfaffian locus of $6\times  6$
matrices is exactly the ideal of $4\times 4$ skew-symmetric minors. This implies that $\cJ_{Q/G}(2)$ is the twisted ideal $\cI_D(2)/\cI_Q(2)$ of $D$ inside $\cO_Q(2)=\cO_Q(Q)$. Let us therefore consider the exact sequence 
\begin{equation}\label{ideals}
 0 \to \cI_Q(2) \to \cI_D(2) \to \cJ_{Q/G}(2)\to 0.
\end{equation} 

By Lemma \ref{lem_I} we have $h^0(G,\cI_D(2))=71$ and in the proof of the same Lemma we showed that $h^i(G,\cI_D(2))=0$, for $i>0$. On the other hand, we have $\cI_Q(2)=\cO_G$. Via the long cohomology exact sequence associated to sequence \eqref{ideals}, we deduce that $h^0(G,\cJ_{Q/G}(2))=70$ and $h^i(G,\cJ_{Q/G}(2))=0$ for $i>0.$ Hence $H'_{Q/G}$ is smooth of dimension 70 at $[Q]$. We have a natural map between Hilbert schemes 
$$\sigma: H'_{Q/G} \to H_D,$$
where $H_D$ is the component of the Hilbert scheme of $G(2,V_8)$ that contains the point $[D]$ defined by $D$. Both spaces have dimension $70$ and are smooth respectively at $[Q]$ and $[D]$ by Corollary \ref{cor_normal_coho}. In order to show that $\sigma$ is dominant,
it is enough to check that the induced morphism of tangent spaces is dominant. This is true because $H^0(G,\cJ_{Q/G}(2))$ and $H^0(\cN_{D/G})$ are both dominated by $H^0(\cI_D(2))$, and the morphism from $\cI_D(2)$ to $\cN_{D/G}$ factorizes through $\cJ_{Q/G}(2)$. This concludes the proof.
\end{proof}

\subsection{Grassmannian self-duality}
Exactly as we constructed the singular quadric hypersurface $D_{Z_1}(v)\subset G(2,V_8)$,
there is another hypersurface $D_{Z_1}(v^\vee)\subset G(2,V_8^\vee)=G(6,V_8)$. Because of Proposition \ref{iso} these two hypersurfaces are projectively isomorphic. But one should also  expect some projective duality statement analogous to Theorem \ref{selfdual}. Of course we cannot refer to classical projective duality, since we want to consider $D_{Z_1}(v)$ and $D_{Z_1}(v^\vee)$ really as hypersurfaces in Grassmannians, not as subvarieties of the 
ambient projective spaces. It turns out that a version of projective duality in this setting 
(and for certain other ambient varieties than Grassmannians) was once proposed in \cite{chaput} (that remained unpublished). We will refer to it as  {\it Grassmannian
duality}.

The idea is the following. Consider, say, a hypersurface $H$ in $G(2,V_8)$ (or any 
Grassmannian, but let us restrict to the case we are interested in). At a smooth point  
$h=[U_2]$ of $H$, the tangent space to $H$ is a hyperplane in $T_hG(2,V_8)=\Hom(U_2,V_8/U_2)$, 
or equivalently, a line in the dual space $\Hom(V_8/U_2,U_2)$. 
If this line is generated by a 
surjective morphism, the kernel of this morphism 
is a four-dimensional subspace of $V_8/U_2$. Equivalently, 
this defines a six-dimensional space $U_6$ such that $U_2\subset U_6\subset V_8$. We get in this way a rational map from $H$ to $G(6,V_8)$, and we can define the Grassmannian 
dual $H^\vee$ as the image of
this rational map. For more details see \cite[section 1.6]{chaput}. 
Chaput has a remarkable 
Biduality Theorem generalizing the classical statement, according to which duality for subvarieties of Grassmannians is an involution \cite[Theorem 2.1]{chaput}.

So this Grassmannian duality is perfectly natural, and we have: 

\begin{theorem}\label{selfdual2}
$D_{Z_1}(v)\simeq D_{Z_1}(v^\vee)$ is Grassmannian self-dual. 
\end{theorem}

\proof Suppose that $U_2$ belongs to $D_{Z_1}(v)$. By definition, this means that there exists 
$U_4\supset U_2$ (unique in general) such that  $$v \in U_2\wedge (\wedge^3V_8)+(\wedge^2U_4)\wedge (\wedge^2V_8).$$
If we mod out by $ \wedge^2U_4$, we get a tensor in $U_2\otimes \wedge^3(V_8/U_4)\simeq U_2\otimes (V_8/U_4)^\vee$, that is, a morphism from $V_8/U_4$ to $U_2$. Generically 
this morphism has full rank, and its kernel  defines some $U_6\supset U_4$. So we get a flag $(U_2\subset U_4\subset U_6)$ such that 
\begin{equation}\label{dec1}
v\in U_2\wedge (\wedge^2U_6)\wedge V_8+(\wedge^2U_4)\wedge (\wedge^2V_8).
\end{equation}

\begin{lemma}
$U_6$ defines a point of $D_{Z_1}(v^\vee)$. 
\end{lemma}

\proof Using adapted basis, one  checks that condition (\ref{dec1}) implies that 
$$v^\vee\in U_6^\perp\wedge (\wedge^2U_2^\perp)\wedge V_8^\vee+
(\wedge^2U_4^\perp)\wedge (\wedge^2V_8^\vee).$$
In particular, $v^\vee$ mod $U_6^\perp$ has rank at most four. \qed 

\begin{lemma}
$U_6$ defines a point of $D_{Z_1}(v)^\vee$. 
\end{lemma}

\proof Using a  basis of $V_8$ adapted to the flag $(U_2\subset U_4\subset U_6)$, we can rewrite relation (\ref{dec1}) in the form 
$$v=e_1\wedge  e_5\wedge e_6\wedge e_7+e_2\wedge  e_5\wedge e_6\wedge e_8+v', \quad v'\in (\wedge^2U_4)\wedge (\wedge^2V_8),$$
where $U_2=\langle e_1,e_2\rangle $ and $U_6=\langle e_1,\ldots ,e_6\rangle $. 
We can describe infinitesimal deformations of $U_2$ by some infinitesimal deformations 
of the vectors in the adapted basis, say $e_i\mapsto e_i+\epsilon \delta_i$, and we 
must keep a similar relation. Modding out by $U_4$, we only remain with the relation 
$$\delta_1\wedge  e_5\wedge e_6\wedge e_7+\delta_2\wedge  e_5\wedge e_6\wedge e_8=0 
\qquad \mathrm{mod} \; U_4,$$
which we can simply rewrite as $\delta_{18}=\delta_{27}$.
This relation describes the tangent hyperplane to $D_{Z_1}(v)$ at $U_2$, as a hyperplane 
in $\Hom(U_2,V_8/U_2)$,  orthogonal to the morphism $e_8^*\otimes e_1-e_7^*\otimes e_2$. 
The kernel of this morphism is $U_6/U_2$, and we are done. \qed

\medskip These two Lemmas together imply that $D_{Z_1}(v^\vee)$ coincides with the 
Grassmannian dual to $D_{Z_1}(v)$. The proof of the Theorem is complete. \qed 

\medskip Note that we can resolve the singularities of $D_{Z_1}(v)$ by considering 
flags $(U_2\subset U_4)$ as before, which gives a subvariety $\tilde{D}_{Z_1}(v)\subset Fl(2,4,V_8)$. By considering the flags $(U_2\subset U_4\subset U_6)$ as in the proof
of the previous statement we obtain a subvariety $D_{Z_1}(v,v^\vee)\subset Fl(2,4,6,V_8)$
that resolves simultaneously the singularities of $D_{Z_1}(v)$
and $D_{Z_1}(v^\vee)$. As for the Coble quartic, we get a diagram
$$\xymatrix@-2ex{&& Fl(2,4,6,V_8) & &\\
Fl(2,4,V_8) & & D_{Z_1}(v,v^\vee)\ar[dr]\ar[dl]\ar@{^{(}->}[u] && Fl(4,6,V_8)\\
 & \tilde{D}_{Z_1}(v)\ar[dl]\ar@{_{(}->}[ul]  \ar@{-->}[rr]& & 
 \tilde{D}_{Z_1}(v^\vee)\ar[dr]\ar@{^{(}->}[ur] &\\
G(2,V_8)\supset  D_{Z_1}(v)  \ar@{-->}[rrrr]^{dD}& & && D_{Z_1}(v^\vee)\subset G(6,V_8) 
}$$
The birational map $\tilde{D}_{Z_1}(v)\dashrightarrow \tilde{D}_{Z_1}(v^\vee)$
must be a flop, resolved by two symmetric contractions.

\medskip \noindent {\it Question.}
Is there a modular interpretation of $D_{Z_1}(v)$ as for the Coble quartic? 
And of this diagram?

\begin{remark} 
Our framework excludes the hyperelliptic genus three curves, but there should be a 
very similar story for these curves. In fact, consider a general pencil of quadrics in $\PP^7=\PP(V_8)$. The eight singular members of the pencil define such a 
hyperelliptic curve $C$. 
It is a special case of the results of \cite{desale-ramanan} that the moduli space 
$\SU_C(2,L)$, for $L$ of odd degree, can be identified with the bi-orthogonal 
Grassmannian, that is the subvariety of $G(2,V_8)$ parametrizing subspaces that 
are isotropic with respect to any quadric in the pencil. On the other hand, the even 
moduli space $\SU_C(2)$ is a double cover of the six-dimensional quadric $\QQ^6$, 
branched over a quartic section which is singular along a copy of the Kummer threefold
of the curve. One expects this quartic to be of Coble type, in the sense that it should be 
the unique quartic section of $\QQ^6$ that is singular along the Kummer of $C$. 
It should also 
be self-dual in a suitable sense, and the whole story should be related to the 
representation theory of $\mathrm{Spin}_8$. We plan to explore these topics in future work. 
\end{remark} 

\bibliography{moduli}

\providecommand{\bysame}{\leavevmode\hbox to3em{\hrulefill}\thinspace}
\providecommand{\MR}{\relax\ifhmode\unskip\space\fi MR }
\providecommand{\MRhref}[2]{%
  \href{http://www.ams.org/mathscinet-getitem?mr=#1}{#2}
}
\providecommand{\href}[2]{#2}
\begin{thebibliography}{BFMT20b}

\bibitem[AB15]{AB15}
Alberto Alzati and Michele Bolognesi, \emph{A structure theorem for
  {${\mathcal{SU}}_C(2)$} and the moduli of pointed rational curves}, J.
  Algebraic Geom. \textbf{24} (2015), no.~2, 283--310.

\bibitem[Bea88]{beauville1}
Arnaud Beauville, \emph{Fibr\'{e}s de rang {$2$} sur une courbe, fibr\'{e}
  d\'{e}terminant et fonctions th\^{e}ta}, Bull. Soc. Math. France \textbf{116}
  (1988), no.~4, 431--448.

\bibitem[Bea91]{beauville}
\bysame, \emph{Fibre bundles of rank two on a curve, the determinant bundle and
  theta functions}, Bull. Soc. Math. Fr. \textbf{119} (1991), no.~3, 259--291
  (French).

\bibitem[Bea03]{beauville-coble}
\bysame, \emph{The {C}oble hypersurfaces}, C. R. Math. Acad. Sci. Paris
  \textbf{337} (2003), no.~3, 189--194.

\bibitem[BFMT20a]{bfmt}
Vladimiro Benedetti, Sara~Angela Filippini, Laurent Manivel, and Fabio
  Tanturri, \emph{Orbital degeneracy loci and applications}, Ann. Sc. Norm.
  Super. Pisa Cl. Sci. (5) \textbf{21} (2020), 169--206.

\bibitem[BFMT20b]{bfmt2}
\bysame, \emph{Orbital degeneracy loci {II}: {G}orenstein orbits}, Int. Math.
  Res. Not. IMRN (2020), no.~24, 9887--9932.

\bibitem[Cha07]{chaput}
Pierre-Emmanuel Chaput, \emph{Dual varieties of subvarieties of homogeneous
  spaces}, arXiv e-print math.AG/0702829, 2007.

\bibitem[Cob61]{Coble}
Arthur~B. Coble, \emph{Algebraic geometry and theta functions.}, American
  Mathematical Society, Providence, R.I.,, 1961, Revised printing.

\bibitem[DR77]{desale-ramanan}
Usha~N. Desale and Sundararaman Ramanan, \emph{Classification of vector bundles
  of rank {$2$} on hyperelliptic curves}, Invent. Math. \textbf{38} (1976/77),
  no.~2, 161--185.

\bibitem[EN67]{EN}
John~A. Eagon and Douglas~G. Northcott, \emph{Generically acyclic complexes and
  generically perfect ideals}, Proc. Roy. Soc. Ser. A \textbf{299} (1967),
  147--172.

\bibitem[FK87]{FK87}
Hubert Flenner and Siegmund Kosarew, \emph{On locally trivial deformations},
  Publ. Res. Inst. Math. Sci. \textbf{23} (1987), no.~4, 627--665 (English).

\bibitem[GK89]{GK89}
Gert-Martin Greuel and Ulrich Karras, \emph{Families of varieties with
  prescribed singularities}, Compos. Math. \textbf{69} (1989), no.~1, 83--110.

\bibitem[GS]{Macaulay2}
Daniel~R. Grayson and Michael~E. Stillman, \emph{Macaulay2, a software system
  for research in algebraic geometry}, Available at
  \url{http://www.math.uiuc.edu/Macaulay2/}.

\bibitem[GSW13]{GSW}
Laurent Gruson, Steven~V. Sam, and Jerzy Weyman, \emph{Moduli of abelian
  varieties, {V}inberg {$\theta$}-groups, and free resolutions}, Commutative
  algebra, Springer, New York, 2013, pp.~419--469.

\bibitem[HR04]{hr}
Jun-Muk Hwang and Sundararaman Ramanan, \emph{Hecke curves and {H}itchin
  discriminant}, Ann. Sci. \'{E}cole Norm. Sup. (4) \textbf{37} (2004), no.~5,
  801--817.

\bibitem[Hwa00]{hwang}
Jun-Muk Hwang, \emph{Tangent vectors to {H}ecke curves on the moduli space of
  rank 2 bundles over an algebraic curve}, Duke Math. J. \textbf{101} (2000),
  no.~1, 179--187.

\bibitem[JPW81]{jpw}
Tadeusz J\'{o}zefiak, Piotr Pragacz, and Jerzy Weyman, \emph{Resolutions of
  determinantal varieties and tensor complexes associated with symmetric and
  antisymmetric matrices.}, Young tableaux and {S}chur functors in algebra and
  geometry ({T}oru\'{n}, 1980), Ast\'{e}risque, 87-88, 1981, pp.~109--189.

\bibitem[KM23]{km}
Akihiro Kanemitsu and Shigeru Mukai, \emph{K3 surfaces of genera $13$ and
  $19$}, in preparation, 2023.

\bibitem[Kol23]{kollar}
János Kollár, \emph{Arthur byron coble, 1878--1966}, arXiv e-prinet
  math.AG/2306.05940, 2023.

\bibitem[KW13]{KWE7}
Witold Kraskiewicz and Jerzy Weyman, \emph{Geometry of orbit closures for the
  representations associated to gradings of lie algebras of types $e_7$}, arXiv
  e-print math.AG/301.0720, 2013.

\bibitem[Man06]{config}
Laurent Manivel, \emph{Configurations of lines and models of {L}ie algebras},
  J. Algebra \textbf{304} (2006), no.~1, 457--486.

\bibitem[MS09]{mok-sun}
Ngaiming Mok and Xiao~Tao Sun, \emph{Remarks on lines and minimal rational
  curves}, Sci. China Ser. A \textbf{52} (2009), no.~4, 617--630.

\bibitem[MTiB20]{mustopa-teixidor}
Yusuf Mustopa and Montserrat Teixidor~i Bigas, \emph{Rational curves on moduli
  spaces of vector bundles}, arXiv e-print math.AG/2007.10511, 2020.

\bibitem[NR75]{nr75}
Mudumbai~S. Narasimhan and Sundararaman Ramanan, \emph{Deformations of the
  moduli space of vector bundles over an algebraic curve}, Ann. of Math. (2)
  \textbf{101} (1975), 391--417.

\bibitem[NR87]{nr}
\bysame, \emph{{$2 \theta$}-linear systems on abelian varieties}, Vector
  bundles on algebraic varieties ({B}ombay, 1984), Tata Inst. Fund. Res. Stud.
  Math., vol.~11, Tata Inst. Fund. Res., Bombay, 1987, pp.~415--427.

\bibitem[Oed22]{oeding}
Luke Oeding, \emph{A translation of "classification of four-vectors of an
  8-dimensional space," by antonyan, l. v. , with an appendix by the
  translator}, arXiv e-print math.AG/2205.09741, 2022.

\bibitem[OPP98]{opp}
William~M. Oxbury, C.~Pauly, and E.~Previato, \emph{Subvarieties of
  {$\mathcal{SU}_C(2)$} and {$2\theta$}-divisors in the {J}acobian}, Trans.
  Amer. Math. Soc. \textbf{350} (1998), no.~9, 3587--3614.

\bibitem[Pal16]{pal}
Sarbeswar Pal, \emph{Moduli of rank 2 stable bundles and {H}ecke curves},
  Canad. Math. Bull. \textbf{59} (2016), no.~4, 865--877.

\bibitem[Pau02]{pauly}
Christian Pauly, \emph{Self-duality of {C}oble's quartic hypersurface and
  applications}, Michigan Math. J. \textbf{50} (2002), no.~3, 551--574.

\bibitem[Ram73]{ramanan}
Sundararaman Ramanan, \emph{The moduli spaces of vector bundles over an
  algebraic curve}, Math. Ann. \textbf{200} (1973), 69--84.

\bibitem[RSSS13]{rsss}
Qingchun Ren, Steven~V. Sam, Gus Schrader, and Bernd Sturmfels, \emph{The
  universal {K}ummer threefold}, Exp. Math. \textbf{22} (2013), no.~3,
  327--362.

\bibitem[Ser06]{Sernesi}
Edoardo Sernesi, \emph{Deformations of algebraic schemes}, Grundlehren Math.
  Wiss., vol. 334, Berlin: Springer, 2006 (English).

\bibitem[Sun05]{sun}
Xiaotao Sun, \emph{Minimal rational curves on moduli spaces of stable bundles},
  Math. Ann. \textbf{331} (2005), no.~4, 925--937.

\bibitem[vGP92]{vGP}
Bert van Geemen and Emma Previato, \emph{Prym varieties and the {V}erlinde
  formula}, Math. Ann. \textbf{294} (1992), no.~4, 741--754.

\bibitem[Vin76]{vinberg}
\`{E}rnest~B. Vinberg, \emph{The {W}eyl group of a graded {L}ie algebra.}, Izv.
  Akad. Nauk SSSR Ser. Mat. (1976), no.~no. 3,, 488--526, 709.

\bibitem[vLCL92]{LiE}
Marc~A.A. van Leeuwen, Arjeh~M. Cohen, and Bert Lisser, \emph{{LiE}, a package
  for {L}ie group computations}, Computer Algebra Nederland, Amsterdam, 1992.

\bibitem[Wey03]{weyman-book}
Jerzy Weyman, \emph{Cohomology of vector bundles and syzygies}, Cambridge
  Tracts in Mathematics, vol. 149, Cambridge University Press, Cambridge, 2003.

\end{thebibliography}

\bibliographystyle{amsalpha}

\footnotesize

\bigskip 
Institut de Mathématiques de Bourgogne, Université de 
Bourgogne et Franche-Comté, 9 Avenue Alain Savary, 21078 Dijon Cedex, {\sc France}.

{\it Email address}: {\tt vladimiro.benedetti@u-bourgogne.fr}

\smallskip 

Institut Montpelli\'erain Alexander Grothendieck, Université de Montpellier, 
 Place  Eugène Bataillon, 34095 Montpellier Cedex 5, {\sc France}.
 
{\it Email address}: {\tt michele.bolognesi@umontpellier.fr}

\smallskip 

Institut de Mathématiques de Bourgogne, Université de 
Bourgogne et Franche-Comté, 9 Avenue Alain Savary, 21078 Dijon Cedex, {\sc France}.

{\it Email address}: {\tt daniele.faenzi@u-bourgogne.fr}

\smallskip 
Institut de Math\'ematiques de Toulouse, 
Paul Sabatier University, 118 route de Narbonne, 31062 Toulouse
Cedex 9, {\sc France}.

{\it Email address}: {\tt manivel@math.cnrs.fr}
\end{document}